\documentclass[reqno]{amsart}

\usepackage{enumitem}
\usepackage{amssymb}
\usepackage{hyperref}

\usepackage{geometry}
\newcommand*{\mailto}[1]{\href{mailto:#1}{\nolinkurl{#1}}}

\newtheorem{theorem}{Theorem}[section]
\newtheorem{definition}[theorem]{Definition}
\newtheorem{lemma}[theorem]{Lemma}

\newtheorem{proposition}[theorem]{Proposition}

\newcommand{\R}{{\mathbb R}}
\newcommand{\Z}{{\mathbb Z}}
\newcommand{\C}{{\mathbb C}}
\newcommand{\D}{{\mathcal D}}
\newcommand{\cA}{{\mathcal A}}
\newcommand{\cR}{{\mathcal R}}

\newcommand{\cQ}{{\mathcal Q}}
\newcommand{\bl}{{\bf l}}
\newcommand{\I}{\mathrm{i}}
\newcommand{\E}{\mathrm{e}}
\DeclareMathOperator{\re}{Re}
\DeclareMathOperator{\im}{Im}

\numberwithin{equation}{section}

\begin{document}

\title[Scattering Properties and Dispersion Estimates for a Discrete Equation]{Scattering Properties and Dispersion Estimates for a One-Dimensional Discrete Dirac Equation}

\author[E.\ Kopylova]{Elena Kopylova}
\address{Faculty of Mathematics\\ University of Vienna\\
Oskar-Morgenstern-Platz 1\\ 1090 Wien\\ Austria}
\email{\mailto{Elena.Kopylova@univie.ac.at}}
\urladdr{\url{http://www.mat.univie.ac.at/~ek/}}

\author[G.\ Teschl]{Gerald Teschl}
\address{Faculty of Mathematics\\ University of Vienna\\
Oskar-Morgenstern-Platz 1\\ 1090 Wien\\ Austria}
\email{\mailto{Gerald.Teschl@univie.ac.at}}
\urladdr{\url{http://www.mat.univie.ac.at/~gerald/}}

\thanks{{\it EK is supported by the Austrian Science Fund (FWF) under Grant No. P 34177}}
\thanks{Math. Nachr. 295, 762--784 (2022)}


\begin{abstract}
We derive dispersion estimates for solutions of a one-dimensional discrete Dirac equations with a potential.
In particular, we improve our previous result, weakening the conditions on the potential. To this end we also provide
new results concerning scattering for the corresponding perturbed Dirac operators which are
of independent interest. Most notably, we show that the reflection and transmission coefficients belong to the Wiener algebra.
\end{abstract}

\maketitle

\section{Introduction}
We are concerned with  the one-dimensional discrete  Dirac equation
\begin{equation} \label{Dir}
 i\dot{\bf w}(t):=\D{\bf w}(t)=(\D_0+Q){\bf w}(t),
\quad {\bf w}_n=(u_n,v_n)\in \C^2,\quad n\in\Z.
\end{equation}
Here the discrete free Dirac operator $\D_0$ is defined by
$$
\D_0=\left(\begin{array}{cc}
m   & d\\
d^* & -m
\end{array} \right),\quad m>0,
$$
where $(du)_n=u_{n+1}-u_n$.
For the real potential $Q$ we assume that 
\begin{equation}\label{Q}
Q_n=\left(\begin{array}{cc}
0   & q_n\\
q_n & 0
\end{array} \right),~~{\rm where}~~q_n\not = 1,\quad n\in\Z,
\end{equation}
is bounded, such that $\D$ gives rise to a bounded self-adjoint operator in $\bl^2(\Z)=\ell^2(\Z)\oplus\ell^2(\Z)$.

In the first part of our article we  show that the scattering matrix of the operator $\D$ 
is in the Wiener algebra (i.e.\ its Fourier coefficients are summable) if  the first moment of the potential   is summable.
We use this result to establish dispersive decay estimates for equation \eqref{Dir}  
under weaker assumption than in our previous results \cite{KT1}.

Let us  introduce the weighted spaces 
$\ell^p_{\sigma}=\ell^p_{\sigma}(\Z)$,
$\sigma\in\R$,  associated with the norm
\begin{equation*}
\Vert u\Vert_{\ell^p_{\sigma}}= \begin{cases} 
\left( \sum_{n\in\Z} (1+|n|)^{p\sigma} |u_n|^p\right)^{1/p}, & \quad p\in[1,\infty),\\
\sup_{n\in\Z} (1+|n|)^{\sigma} |u_n|, & \quad p=\infty,
\end{cases}
\end{equation*}
and the case $\sigma=0$ corresponds to the standard spaces $\ell^p_0=\ell^p$  without weight.
Denote $\bl^p_{\sigma}=\ell^p_{\sigma}\oplus \ell^p_{\sigma}$ and $\bl^p=\ell^p\oplus \ell^p$.

We recall that under the condition $q\in\ell^1_1$, 
the spectrum of $\D$ consists of a purely absolutely continuous part
$\Gamma=(-\sqrt{4+m^2}, -m)\cup(m,\sqrt{4+m^2})$, plus a finite number of eigenvalues located in $\R\setminus\overline\Gamma$.
In addition, there could be resonances at the edges $\omega=\pm m,\pm\sqrt{4+m^2}$ of the continuous spectrum  (see \cite{KT1}).
\smallskip

As our first main result, we  prove the following $\bl^1\to \bl^\infty$ decay
\begin{equation}\label{fullp}
\Vert \E^{-\I t\D}P_{c}\Vert_{\bl^1\to \bl^\infty}=\mathcal{O}(t^{-1/3}),\quad t\to\infty,
\end{equation}
under the assumptions $q\in\ell^1_1$.
Here $P_{c}$ is the orthogonal projection in $\bl^2$ onto the
continuous spectrum of $\D$.  
\smallskip

Secondly, we establish decay in $\bl^2_\sigma\to \bl^2_{-\sigma}$ with $\sigma>1/2$:
\begin{equation}\label{as1-new}
\Vert \E^{-\I t\D} P_c\Vert_{\bl^2_\sigma\to \bl^2_{-\sigma}}=\mathcal{O}(t^{-1/2}),\quad t\to\infty.
\end{equation}
Let us emphasize that we do not require additional decay of $q$ for \eqref{fullp}--\eqref{as1-new}
in the case when the edges of the continuous spectrum  are resonances.

In the remaining results we restrict ourselves to the non-resonance case. In this case, for $q\in\ell^1_2$, we show that
\begin{equation}\label{as-new}
\Vert \E^{-\I t\D}P_c\Vert_{\bl^1_1\to \bl^\infty_{-1}}=\mathcal{O}(t^{-4/3}),
\quad t\to\infty,
\end{equation}
and
\begin{equation}\label{as2-new}
\Vert \E^{-\I t\D} P_c\Vert_{\bl^2_\sigma\to \bl^2_{-\sigma}}=\mathcal{O}(t^{-3/2}),\quad t\to\infty,\quad\sigma>3/2.
\end{equation}

The dispersion estimates \eqref{fullp}--\eqref{as1-new}
have been established in our previous paper  \cite{KT1} under the assumption $q\in \ell^1_2$ 
in the non-resonance case,  and  
under the more restrictive condition $q\in \ell^1_3$ in the resonance case.  Moreover, in  \cite{KT1},  we required $q\in \ell^1_3$ for 
the asymptotics \eqref{as-new}--\eqref{as2-new} to hold in the non-resonance case. 

To show  that the extra decay of $q$  is not necessary,  we extend the approach of \cite{EKT, EKTM}, 
introduced in the context of discrete and continuous Schr\"odinger equations, which relies on a refined version of an
old  result of Guseinov \cite{Gus}. 
Namely, we prove that the transmission and reflection coefficients  $T(\theta)$ and $R^\pm(\theta)$ belong to
the Wiener algebra $\cA$. Let us note that in the half-line case the analogous result for the scattering data
is well known (cf.\ Problem~3.2.1 in \cite{Mar}) and was used by Weder \cite{W} to prove
a corresponding result in the half-line case.

Our approach can be summarized as follows:
To prove that $T(\theta), R^\pm(\theta)\in \cA$, we  first compute  the Fourier coefficients of the Jost solutions ${\bf h}^{\pm}(\theta)=(h^{\pm}_1(\theta), h^{\pm}_1(\theta))$. 
The main difficulty here is the presence of the factors 
 $\lambda\pm m$,  where $\lambda=\sqrt{m^2+2-\E^{\I\theta}-\E^{-\I\theta}}$, in the Green function
 (formula \eqref{Green0} below).
This implies  that the Fourier series for ${\bf h}^{\pm}(\theta)$ contain all powers of $\E^{ \I\theta}$ 
contrary to the Schr\"odinger  case,  where corresponding  Fourier series  contain nonnegative powers only.
Nevertheless,  we obtain the Fourier series only with nonnegative powers of $\E^{ \I\theta}$  for  $(h^{\pm}_1(\theta), (m+\lambda)h^{\pm}_1(\theta))$
in the case  $\lambda>0$ (and for  $((m-\lambda)h^{\pm}_1(\theta), h^{\pm}_1(\theta))$ in the case $\lambda<0$), see  formulas \eqref{B} and \eqref{B-} below.
 
Using these Fourier series, we then derive  the Gelfand--Levitan--Marchenko equations \eqref{GLM-eq}--\eqref{GLM-2}
for  the Fourier coefficients ${\mathcal F}_n^{\pm}$ of $R^\pm(\theta)$. The extra factors  $\lambda\pm m$ cancel and do not appear in these equations.
Moreover, these equations have a standard form and provide estimates for ${\mathcal F}_n^{\pm}$ similar to the estimates of \cite[\S 10]{tjac}, (see also \S 3.5 in \cite{Mar}). 

To prove the decay estimates \eqref{fullp}--\eqref{as2-new} we apply the spectral Fourier--Laplace representation 
\[
   \E^{-\I t\D}P_{c}=\frac 1{2\pi \I}\int\limits_{\Gamma} \E^{-\I t\lambda}( \cR(\lambda+\I 0)- \cR(\lambda-\I 0))\,d\lambda.
\]
Expressing  the kernels of the resolvents $\cR(\lambda\pm\I 0)$ in terms of  Jost solutions and
using the scattering relation \eqref{scat-rel},  we get oscillatory integrals with amplitudes from the Wiener algebra $\cA$. 
This integral representation implies \eqref{fullp}--\eqref{as2-new} by a suitable version of the van der Corput lemma.

We remark that the derivation of the Gelfand--Levitan--Marchenko equations for arbitrary self-adjoint perturbations $Q$ remains an open problem.
\section{Jost solutions}
\label{Jost-func}
Here we recall some spectral  properties of equation \eqref{Dir} which we obtain in \cite{KT1} using the Jost solutions.
Set $\Gamma_+=(m,\sqrt{4+m^2})$, and let $\Xi_{+}=\{\lambda\in\C\setminus\overline\Gamma_+,~\re\lambda\ge 0\}$.
For any $\lambda\in \overline\Xi_+$,  we consider the Jost solutions ${\bf w}=(u, v)$ of
\begin{equation} \label{D1}
{\D}{\bf w}=\lambda{\bf w}
\end{equation}
defined by the boundary conditions
\begin{equation}\label{J2}
{\bf w}_n^{\pm}(\theta)=
\left(\begin{array}{ll}
u_n^{\pm}(\theta)\\
v_n^{\pm}(\theta)
\end{array}\right)
\to
\left(\begin{array}{ll}
1\\
\alpha_{\mp}(\theta)
\end{array}\right)\E^{\pm\I\theta n},\quad n\to\pm\infty, \qquad  \alpha_{\pm}(\theta):=\frac{\E^{\pm\I\theta}-1}{m+\lambda}.
\end{equation}
Here  $\theta=\theta(\lambda)\in\overline\Sigma:=\{-\pi\le \re\theta\le\pi,~\im\theta\ge 0\}$ is the solution to
\[
2-2\cos\theta=\lambda^2-m^2.
\]
These boundary condition  arise naturally in \eqref{D1}  with
$Q \equiv 0$.
For nonzero $Q$ with $q\in \ell_1^1$,
the Jost solutions exists everywhere in $\overline\Xi_+$, but for 
$q\in \ell^1$ they only exist away from the edges of the continuous spectrum.
Introduce
\begin{equation}\label{Jostcut}
{\bf h}^\pm_n(\theta)=\E^{\mp \I n \theta}{\bf w}^\pm_n(\theta)
\end{equation}
and set
\[
\overline\Sigma_M:=\{\theta\in\overline\Sigma:\im\theta\le M \},\qquad\overline\Sigma_{M,\delta}:=\{\theta\in\overline\Sigma_{M}:\, |\E^{\I\theta} \pm 1|>\delta\},
\quad M\ge 1, \quad 0<\delta<\sqrt 2.
\]
\begin{lemma}\label{Jost-sol} (see \cite [Proposition 3.1]{KT1})\\
(i) Let $q\in\ell^1_s$ with $s=0,1,2$. Then the functions ${\bf h}^\pm_n(\theta)$ can be differentiated $s$ times on $\overline\Sigma_{M,\delta}$, and  the following estimates hold:
\begin{equation}\label{dh-est}
|\frac{d^p}{d \theta^p}{\bf h}^\pm_n(\theta) |\le C(M,\delta)\max ((\mp n)|n|^{p-1}, 1), \quad n\in\Z,\quad 0\le p\le s,\quad \theta\in \overline\Sigma_{M,\delta}.
\end{equation}
\\
(ii) If additionally  $q\in \ell_{s+1}^1$, then ${\bf h}^\pm_n(\theta)$ can be differentiated $s$ times on $\overline\Sigma_M$, and the following estimates hold:
\begin{equation}\label{Jostderiv}
|\frac{d^{p}}{d \theta^{p}}{\bf h}^\pm_n(\theta) |\le C(M)\max ((\mp n)|n|^{p}, 1),\quad n\in\Z,\quad 0\le p\le s,\quad\theta\in\overline\Sigma_M.
\end{equation}
\end{lemma}
In the case $q\in \ell^1$ Proposition \ref{Jost-sol} (i) implies, in particular, that for any
$\theta\in\overline\Sigma\setminus\{0;\pm\pi\}$
we have the estimate $|{\bf h}_n^\pm(\theta)|\leq C(\theta)$ for all $n\in\Z$, where $C(\theta)$ can be chosen uniformly in compact subsets of $\overline{\Sigma}$ avoiding the band edges. Together with \eqref{Jostcut} this implies
\begin{equation}\label{De0}
| {\bf w}^{\pm}_n(\theta) |\le C(\theta) \E^{\mp \im(\theta)n}\, ,
\quad \theta\in\overline\Sigma\setminus\{0;\pm\pi\},\quad n\in\Z.
\end{equation}
Denote by  $W({\bf w}^1,{\bf w}^2)$  the Wronskian determinant of any two solutions
${\bf w}^1$ and ${\bf w}^2$ to  (\ref{D1}):
\begin{equation}\label{W-def}
W({\bf w}^1,{\bf w}^2):=\left|\begin{array}{llll}
 u^1_n & u^2_n\\
 v^1_{n+1}    &  v^2_{n+1}
\end{array}\right|.
\end{equation}
It is easy to check that $W({\bf w}^1,{\bf w}^2)$ is independent of $n\in\Z$
for arbitrary solutions ${\bf w}^1$ and ${\bf w}^2$ of \eqref{D1}.
Denote
\[
W(\theta)=W({\bf w}^+(\theta),{\bf w}^-(\theta)).
\]
\begin{definition}
For $\lambda\in\{m,\sqrt{4+m^2}\}$
any nonzero solution ${\bf w}\in\bl^{\infty}$  of the equation  $\D{\bf w}=\lambda {\bf w}$
is called a resonance function,
and in this case  $\lambda$  is called a resonance.
\end{definition}
\begin{lemma}\label{W} (see \cite[Lemmas 4.1 and 4.4]{KT1})\\
{\it i)}  Let $q\in\ell^1$. Then $W(\theta)\not =0$ for $\theta\in(-\pi,0)\cup(0,\pi)$.\\
{\it ii)} 
Let $q\in\ell^1_1$. Then $\lambda=m$ (or~$\lambda=\sqrt{4+m^2}$) is a resonance
if and only if $W(0)=0$ (or $W(\pi)=0$).
\end{lemma}
Given the Jost solutions, we can express  the resolvent $\cR(\lambda):=({\mathcal D}-\lambda)^{-1}$.
The method of variation of parameters gives:
\begin{lemma}\label{cR-rep}
Let $q\in\ell^1$. Then
for any $\lambda\in  \Xi_+$, the operators
$\cR(\lambda):\bl^2 \to \bl^2$ 
can be represented by the matrix elements  as follows
\begin{equation}\label{RJ-rep}
[\cR(\lambda)]_{n,k}=\frac{1}{W(\theta(\lambda))}
\begin{cases}
{\bf w}_n^+(\theta(\lambda))\otimes{\bf w}_k^-(\theta(\lambda)),\quad k\le n,\\
{\bf w}_n^-(\theta(\lambda))\otimes{\bf w}_k^+(\theta(\lambda)),\quad k\ge n,
\end{cases}
\end{equation}
where
\[
{\bf w}_k^1\otimes{\bf w}_n^2=\left(\begin{array}{cccc}
u_k^1u_n^2& v_{k+1}^1u_n^2\\
u_k^1v_n^2& v_{k+1}^1v_n^2
\end{array}\right),\qquad
\cR(\lambda){\bf w}[n]=\sum\limits_{k=-\infty}^{\infty}
[\cR(\lambda)]_{k,n}\left(\begin{array}{cc}
u_k\\v_{k+1}
\end{array}\right).
\]
\end{lemma}
The representations \eqref{RJ-rep},  the fact that $W(\theta)$ does not vanish for $\lambda\in\Gamma_+$, 
and the bound \eqref{De0} imply the 
limiting absorption principle for the perturbed  Dirac equation.
\begin{lemma}\label{BV}  (see \cite [Lemma 5.2]{KT1})
Let $q\in\ell^1$. Then the convergence
\begin{equation}\label{esk}
    \cR(\lambda\pm \I\varepsilon)\to \cR(\lambda\pm \I0),\quad\varepsilon\to 0+,\quad \lambda\in\Gamma_+
\end{equation}
holds in ${\mathcal L}(\bl^2_\sigma,\bl^2_{-\sigma})$ with $\sigma>1/2$. Here
\begin{equation}\label{RJ1-rep}
[\cR(\lambda\pm \I0)]_{n,k} = \frac{1}{W(\theta_{\pm})} \begin{cases}
{\bf w}_n^+(\theta_{\pm})\otimes {\bf w}_k^-(\theta_{\pm}) \;\; {\rm for} \;\; k \le n, \\
{\bf w}_k^+(\theta_{\pm})\otimes {\bf w}_n^-(\theta_{\pm}) \;\; {\rm for} \;\; k\ge n, 
\end{cases}
\end{equation}
with 
\[
\theta_+:=\theta(\lambda^2-m^2+\I 0)\in [0,\pi], \qquad\theta_-:=\theta(\lambda^2-m^2-\I 0)\in [-\pi,0].
\]
\end{lemma}
\section{Fourier properties of ${\bf h}^\pm_n(\theta)$}
\label{Fourier-sect}
Green's functions  $G^{\pm}(n,\theta)$  of equation \eqref{D1} read: 
\begin{equation}\label{Green0}
G^{\pm}(n,\theta)=\begin{cases}
\frac{(m+\lambda)}{2\I\sin\theta}\left(\begin{array}{cccc}
\E^{\pm\I\theta n}  -\E^{\mp\I\theta n}& 
\alpha_{\pm}\E^{\pm\I\theta n}-\alpha_{\mp}\E^{\mp\I\theta n}\\
\alpha_{\mp}\E^{\pm\I\theta n} -\alpha_{\pm}\E^{\mp\I\theta n}      
&(\E^{\pm\I\theta n}-\E^{\mp\I\theta n})\frac{\lambda-m}{m+\lambda}
\end{array}\right),\quad \mp n\ge 1,\\
0,\quad \mp n\le -1,\end{cases}
\end{equation}
and
\[
G^{+}(0,\theta)=\left(\begin{array}{cc}
0 &  0 \\
-1 &0
\end{array}\right),\qquad
G^{-}(0,\theta)=\left(\begin{array}{cc}
0     & -1 \\
0        & 0 
\end{array}\right),
\]
so that
\[
\left(\begin{array}{cccc}
m-\lambda &  d \\
d^*        & -(m+\lambda)
\end{array}\right)G^{\pm}(\cdot,\theta)[n]=\left(\begin{array}{cccc}
1&0\\0&1
\end{array}\right)\delta_{n0},\quad n\in\Z.
\]

Applying  Green's function representation, we obtain
\[
{\bf w}_n^{\pm}(\theta)=\left(\begin{array}{ll}
1\\
\alpha_{\mp}(\theta)
\end{array}\right)\E^{\pm\I\theta n}-G^{\pm}(0,\theta)Q_n{\bf w}_n^{\pm}(\theta)-\sum\limits_{k=n\pm 1}^{\pm\infty}
G^{\pm}(n-k,\theta)Q_k{\bf w}_k^{\pm}(\theta).
\]
Substituting ${\bf w}_n^{\pm}(\theta)={\bf h}_n^{\pm}(\theta)\E^{\pm \I\theta n}$, we get 
\begin{equation}\label{JR0}
A_n^{\pm}{\bf h}_n^{\pm}(\theta)=\left(\begin{array}{ll}
1\\
\alpha_{\mp}(\theta)
\end{array}\right)+\sum\limits_{k=n\pm1}^{\pm\infty}
\tilde G^{\pm}(k-n,\theta)Q_k{\bf h}_k^{\pm}(\theta),
\end{equation}
where
\[
\tilde G^{\pm}(l,\theta)=\frac{(m+\lambda)}{2\I\sin\theta}\left(\begin{array}{cccc}
\E^{\pm 2\I\theta l} -1& \alpha_{\mp}\E^{\pm 2\I\theta l}-\alpha_{\pm}\\
\alpha_{\pm}\E^{\pm 2\I\theta l} -\alpha_{\mp}      
&(\E^{\pm 2\I\theta l}-1)\frac{\lambda-m}{m+\lambda}
\end{array}\right),\quad \pm l\ge 1,
\]
\[
A_n^+=\left(\begin{array}{cccc}
1 & 0\\
0          & 1-q_n
\end{array}\right),\quad A_n^-=\left(\begin{array}{cccc}
1 -q_n     & 0\\
0       & 1
\end{array}\right).
\]
Representation \eqref{JR0} implies
\begin{proposition}\label{Fh}
Let $q\in\ell^1_1$. Then the Jost solutions ${\bf h}^{\pm}$ are given by
\begin{equation}\label{B}
A_n^{\pm}{\bf h}^\pm_n(\theta) = \left(\begin{array}{ll}
1\\
\alpha_{\mp}(\theta)
\end{array}\right)
+\sum_{k=\mp 1}^{\pm\infty} \left(\begin{array}{cc}
a^\pm_{n,k}\\ \frac{b^\pm_{n,k}}{\lambda+m}\end{array}\right)
\E^{\pm \I k \theta},
\end{equation}
where 
\begin{equation}\label{est3}
|a^\pm_{n,k}|, | b^\pm_{n,k}|
\le  C^\pm_n  \sum_{l=n\pm 1+[k/2]}^{\pm\infty}( |q_l|+\frac{ |q_l|}{|1-q_l|}).
\end{equation}
Moreover,
\begin{equation}\label{estC}
\quad C^\pm_n\le C^{\pm},\quad \mbox{if}\  \pm n\geq 0.
\end{equation}
\end{proposition}
\begin{proof}
Substituting \eqref{B} into \eqref{JR0} and setting $z= \E^{\I  \theta}$,  we  obtain, formally,
\begin{align}\label{sum}
&\!\!\sum_{k=\mp 1}^{\pm \infty} \left(\!\begin{array}{cc}
a_{n,k}^{\pm}\\ \frac{b_{n,k}^{\pm}}{\lambda+m}\end{array}\!\right) z^{\pm k}
=\sum\limits_{p=n\pm 1}^{\pm\infty}
\tilde G^{\pm}(p-n,\theta)Q_p(A_p^{\pm})^{-1}\Big[\left(\!\!\begin{array}{ll}
1\\
\alpha_{\mp}(\theta)\end{array}\!\!\right)
+ \sum_{r=\mp 1}^{\pm\infty}\left(\!\begin{array}{cc}
a_{p,r}^{\pm}\\ \frac{b_{p,r}^{\pm}}{\lambda+m}\end{array}\!\right) z^{\pm r}\Big],
\end{align}
where
\begin{equation}\label{QA}
Q_p(A_p^{+})^{-1}=\left(\begin{array}{cccc}
0 & \tilde q_p
\\
q_p & 0
\end{array}\right), \quad 
Q_p(A_p^{-})^{-1}=\left(\begin{array}{cccc}
0 & q_p
\\
\tilde q_p & 0
\end{array}\right),
\quad \tilde q_p:=\frac{q_p}{1-q_p}.
\end{equation}
{\it Step i)} First we consider the "$+$" case and  represent $\tilde G^{+}(n,\theta)$, $n\ge 1$,  as the sum:
\[
\tilde G^{+}(n,\theta)=
\sum\limits_{j=0}^{2n}(-1)^{j}\left(\!\begin{array}{cccc}
0 & 0
\\
1  & 0
\end{array}\!\right)z^{j}
+\sum\limits_{j=1}^{2n-1}(-1)^{j}\left(\!\begin{array}{cccc}
0 & 1
\\
0  & 0
\end{array}\!\right)z^{j}+\sum\limits_{j=1}^{n}\left(\!\begin{array}{cccc}
\lambda+m & 0
\\
0  & \lambda-m \end{array}\!\right)z^{2j-1}.
\]
Substituting this expression into \eqref{sum} and omitting the ``+"  sign,  we obtain
\begin{align}\nonumber
&\sum_{k=\mp 1}^{\infty} \left(\!\begin{array}{cc}a_{n,k}\\ \frac{b_{n,k}}{\lambda+m}\end{array}\!\right) z^k
=\sum\limits_{p=n+1}^{\infty}\Big[\sum\limits_{j=0}^{2(p-n)}(-1)^{j}z^j
\left(\!\begin{array}{cccc}
0 & 0
\\
0  & \tilde q_p
\end{array}\!\right)
+\sum\limits_{j=1}^{2(p-n)-1}(-1)^{j}z^j
\left(\!\begin{array}{cccc}
 q_p & 0
\\
0  & 0
\end{array}\!\right)\\
\label{bigeq}
&+\sum\limits_{j=1}^{p-n}z^{2j-1}
\left(\!\begin{array}{cccc}
0 & (\lambda+m)\tilde q_p
\\
(\lambda-m)q_p  & 0
\end{array}\!\right)
\Big]\Big[\left(\!\begin{array}{ll}
1\\
\frac{z^{-1}-1}{\lambda+m}\end{array}\!\right)+\sum_{r=\mp}^{\infty}\left(\!\begin{array}{cc}
a_{p,r}\\ \frac{b_{p,r}}{\lambda+m}\end{array}\!\right)\! z^r\Big].
\end{align}
Using $(\lambda- m)(\lambda+ m)=2-z-z^{-1}$,  we rewrite \eqref{bigeq} for the first and second line separately:
\begin{align*}
\sum_{k=-1}^{\infty} a_{n,k}z^k
&=\sum\limits_{p=n+1}^{\infty}q_p\Big[1+\!\sum_{r=-1}^{\infty}a_{p,r}z^r \Big]\sum\limits_{j=1}^{2(p-n)-1}(-1)^{j}z^j \\
&+\sum\limits_{p=n+1}^{\infty}\tilde q_p\Big[\frac 1z-1+\!\sum_{r=-1}^{\infty}b_{p,r}z^r\Big]\sum\limits_{j=1}^{p-n}z^{2j-1},\\\\
\sum_{k=-1}^{\infty} b_{n,k}z^k
&=(2-\frac 1z-z)\sum\limits_{p=n+1}^{\infty}q_p\Big[1+ \sum_{r=-1}^{\infty}a_{p,r}z^r\Big]\sum\limits_{j=1}^{p-n}z^{2j-1}\\ 
&+\sum\limits_{p=n+1}^{\infty}\!\tilde q_p\Big[\frac 1z-1+\sum_{r=-1}^{\infty}b_{p,r}z^r\Big]\sum\limits_{j=0}^{2(p-n)}(-1)^{j}z^j.
\end{align*}
Equating the coefficients of equal powers of $z$, we obtain
\begin{equation}\label{a-1}
a_{n,-1}=0,\quad a_{n,0}=\sum\limits_{p=n+1}^{\infty}\!\!\tilde q_p(1+b_{p,-1}),
\end{equation}
\begin{equation}\label{b-1}
b_{n,-1}=\sum\limits_{p=n+1}^{\infty}\!\!\tilde q_p(1+b_{p,-1}),\quad
b_{n,0}=-\sum\limits_{p=n+1}^{\infty} (q_p[1+a_{p,0}]+\tilde q_p[2+b_{p,-1}-b_{p,0}]),
\end{equation} 
and for   $k\ge 1$, 
\begin{equation}\label{a-k}
a_{n,k}=(-1)^{k}\!\!\sum\limits_{p=n+1+[\frac{k}2]}^{\infty}(q_p+\tilde q_p)
+\sum\limits_{r=0}^{k-1}(-1)^{k+r}\!\!\sum\limits_{p=n+1+[\frac{k-r}2]}^{\infty} \!\!\!q_p a_{p,r}
+\sum\limits_{r=-1}^{[\frac {k-1}2]}\sum\limits_{p=n-r+[\frac {k+1}2]}^{\infty}\!\!\tilde q_p b_{p,f_k(r)},
\end{equation}
\begin{eqnarray}\nonumber
\!\!\!\!\!b_{n,k}&=&(-1)^{k+1}\sum\limits_{p=n+[\frac {k+1}2]}^{\infty}2(q_p+\tilde q_p) +\sigma_kq_{n+\frac{k}2}
+\sum\limits_{r=0}^{k-1}(-1)^{k+r+1}\sum\limits_{p=n+[\frac {k-r+1}2]}^{\infty}2q_p a_{p,r}\\
\label{b-k}
&-&\sum\limits_{r=0}^{k-1}\sigma_{k+r}q_{\frac{2n+k-r}2} a_{\frac{2n+k-r}2,r}
+\sum\limits_{r=-1}^{k-1}(-1)^{r+k}\sum\limits_{p=n+[\frac{k-r+1}2]}^{\infty}\tilde q_p b_{p,r},
\end{eqnarray}
where 
\[
\sigma_k=\begin{cases} 0& \text{for odd } k, \\
1 & \text{for even } k,\end{cases}\qquad
f_k(r)=\begin{cases} 2r & \text{for odd } k, \\
2r+1 & \text{for~ even } k. \end{cases}
\]
These equations  are solved by adapting the iteration of \cite{DT}:
\[
a_{n,k}=\sum_{j=0}^{\infty} a_{j,n,k},\quad\quad
~~~b_{n,k}=\sum_{j=0}^{\infty} b_{j,n,k},
\]
where
\[
a_{0,n,0}=-b_{0,n,-1}=\sum\limits_{p=n+1}^{\infty}\tilde q_p,\quad\quad b_{0,n,0}=-\sum\limits_{p=n+1}^{\infty}(q_p+2\tilde q_p),
\]
\[
a_{0,n,k}=(-1)^{k}\!\!\!\sum\limits_{p=n+1+[\frac k2]}^{\infty}(q_p+\tilde q_p),\quad
b_{0,n,k}=(-1)^{k+1}\!\!\!\sum\limits_{p=n+[\frac {k+1}2]}^{\infty}2(q_p+\tilde q_p) +\sigma_kq_{n+\frac{k}2}, \quad   k\ge 1,
\]
and for $j\ge 0$,
\[
a_{j+1,n,1}=-b_{j+1,n,0}=\sum\limits_{p=n+1}^{\infty}\tilde q_p b_{j,p,0},\quad b_{j+1,n,1} =-\sum\limits_{p=n+1}^{\infty}(q_p a_{j,p,1}+\tilde q_p[b_{j,p,0}-b_{j,p,1}]),
\]
\[
a_{j+1,n,k}=\sum\limits_{r=1}^{k-1}(-1)^{k+r}\sum\limits_{p=n+1+[\frac{k-r}2]}^{\infty}q_p a_{j,p,r}
+\sum\limits_{r=0}^{[\frac {k-1}2]}\sum\limits_{p=n-r+[\frac{k+1}2]}^{\infty}\tilde q_p b_{j,p,f_k(r)},\quad k\ge 1,
\]
\[
b_{j+1,n,k}=\sum\limits_{r=1}^{k-1}(-1)^{k+r+1}\sum\limits_{p=n+[\frac {k-r+1}2]}^{\infty}2q_p a_{j,p,r}
+\sum\limits_{r=0}^{k-1}(-1)^{r+k}\sum\limits_{p=n+[\frac{k-r+1}2]}^{\infty}\tilde q_p b_{j,p,r} ,\quad k\ge 1.
\]
Now we define the functions
\[
\eta(n)=\max\{\sum_{k=n}^{\infty}|q_k|,\sum_{k=n}^{\infty}|\tilde q_k|\},
\quad\gamma(n)=\max\{\sum_{k=n}^{\infty}(k-n)|q_k|,\sum_{k=n}^{\infty}(k-n)|\tilde q_k|\}.
\]
It is  obvious that
$|a_{0,n,k}|+|b_{0,n,k}| \le 2\eta(n+1+[k/2])$.
Moreover, one can show as in \cite[Lemma 3]{DT} that
\[
|a_{j,n,k}|+|b_{j,n,k}|\le \frac{(2\gamma(n))^j}{j!}\eta(n+1+[k/2]).
\]
Hence,   \eqref{est3} with $C_n^{+}=\E^{2\gamma(n)}$ follows.
\smallskip\\
{\it Step ii)}
It is easy to check that in the ``$-$'' case, we obtain similarly  to \eqref{bigeq},
\begin{align*}
&\sum_{k=0}^{-\infty} \left(\!\begin{array}{cc}
a_{n,k}^-\\ \frac{b_{n,k}^-}{\lambda+m} \end{array}\!\right)z^{-k}=\sum\limits_{p=n-1}^{-\infty}\Big[\sum\limits_{j=-1}^{2(p-n)+1}(-1)^{j}z^{-j}
\left(\!\begin{array}{llll}
0 &0\\
0& q_p\end{array}\!\right)+\sum\limits_{j=0}^{2(p-n)}(-1)^{j}z^{-j}
\left(\!\begin{array}{llll}
\tilde q_p &0\\
0& 0\end{array}\!\right)\\
\nonumber
&+\sum\limits_{j=-1}^{p-n}z^{-2j-1}
\left(\!\begin{array}{cccc}
0 & (\lambda+m)q_p
\\
(\lambda-m)\tilde q_p  & 0
\end{array}\!\right)
\Big]\Big[\left(\!\begin{array}{cc}
1\\
\frac{z-1}{\lambda+m}\end{array}\!\right)+ \sum_{r=0}^{-\infty}\left(\!\begin{array}{cc}
a_{p,r}^-\\ \frac{b_{p,r}^-}{\lambda+m}\end{array}\!\right) z^{-r}\Big].
\end{align*}
This  is equivalent to the system 
\begin{align*}
\sum_{k=0}^{-\infty} a_{n,k}^-z^{-k}
&=\sum\limits_{p=n-1}^{-\infty}\tilde q_p\Big[1+\sum_{r=0}^{-\infty}a_{p,r}^-z^{-r} \Big]\!\sum\limits_{j=0}^{2(p-n)}(-1)^{j}z^{-j }\\
&+\sum\limits_{p=n-1}^{-\infty} q_p\Big[z-1+\sum_{r=0}^{-\infty}b_{p,r}^-z^{-r}\Big]\sum\limits_{j=-1}^{p-n}z^{-2j-1},\\\\
\sum_{k=0}^{-\infty} b_{n,k}^-z^{-k}
&=\sum\limits_{p=n-1}^{-\infty}\tilde q_p\Big[1+ \sum_{r=0}^{-\infty}a_{p,r}^-z^{-r}\Big]\sum\limits_{j=-1}^{p-n}(2-z^{-1}-z)z^{-2j-1} \\
&+\sum\limits_{p=n-1}^{-\infty}q_p\Big[z-1+\sum_{r=0}^{-\infty}b_{p,r}^-z^{-r}\Big]\sum\limits_{j=-1}^{2(p-n)+1}(-1)^{j}z^{-j}.
\end{align*}
Equating the coefficients of equal powers of $z$, we obtain  equations
for $a_{n,k}^-$ and $b_{n,k}^-$ similar to the equations \eqref{a-1}--\eqref{a-k}.
In particular, we get
\begin{equation}\label{ab2}
a^-_{n,0}=-b^-_{n,0}=\sum\limits_{p=n-1}^{-\infty}\tilde q_p[1+a^-_{p,0}].
\end{equation}
\end{proof}
\section{The Gelfand--Levitan--Marchenko equations}
\label{A}
The following formula  is obtained by means of  simple calculations:
\begin{lemma}\label{lA1} 
For any ${\bf w}^1=(u^1,v^1)$, ${\bf w}^2=(u^2,v^2)$, 
\begin{equation}\label{G1}
\sum\limits_{j=m}^n \Big({\bf w}^1_j\cdot(\D{\bf w}^2)_j-(\D{\bf w}^1)_j\cdot {\bf w}^2_j\Big)
=-W_n({\bf w}^1,{\bf w}^2)+W_{m-1}({\bf w}^1,{\bf w}^2),
\end{equation}
where ${\bf w}^1_j\cdot{\bf w}^2_j=u^1_ju^2_j+v^1_jv^2_j$, and
$W_j({\bf w}^1,{\bf w}^2)=u^1_jv^2_{j+1}-u^2_jv^1_{j+1}$.
\end{lemma}
Let now ${\bf w}^1$ and  ${\bf w}^2$ be solutions to \eqref{D1}. Then
\[
\frac{d}{d\lambda}(\D-\lambda){\bf w}^k=(\D-\lambda)\frac{d}{d\lambda}{\bf w}^k-{\bf w}^k=0,\quad k=1,2,
\]
and \eqref{G1} implies 
\[
-W_n({\bf w}^1,\frac{d}{d\lambda}{\bf w}^2)+W_{m-1}({\bf w}^1,\frac{d}{d\lambda}{\bf w}^2)
=\sum\limits_{j=m}^n \Big({\bf w}^1_j\cdot(\D\frac{d}{d\lambda}{\bf w}^2)_j-(\D{\bf w}^1)_j\cdot \frac{d}{d\lambda} {\bf w}^2_j\Big)\
=\sum\limits_{j=m}^n {\bf w}^1_j\cdot{\bf w}^2_j.
\]
Using this formula, we obtain 
\begin{lemma}\label{lA3} (cf. \cite[Lemma 2.4]{tjac})
Let ${\bf w}^{\pm}(\lambda)$  be solutions to \eqref{D1}, square summable  near $\pm\infty$. Then
\begin{eqnarray}\nonumber
W_n({\bf w}^{+}(\lambda),\frac{d}{d\lambda}{\bf w}^{+}(\lambda))
&=&\sum\limits_{j=n+1}^{\infty}{\bf w}^{+}_j(\lambda)\cdot{\bf w}^{+}_j(\lambda),\\
\label{G3}
W_n({\bf w}^{-}(\lambda),\frac{d}{d\lambda}{\bf w}^{-}(\lambda))&=&-\sum\limits_{j=-\infty}^n{\bf w}^{-}_j(\lambda)\cdot{\bf w}^{-}_j(\lambda).
\end{eqnarray}
\end{lemma}

Let now $\lambda_l$ be an isolated eigenvalue of $\D$. In this case
$W({\bf w}^{+}(\lambda_l),{\bf w}^{-}(\lambda_l))=0$, and hence
${\bf w}^{\pm}(\lambda_l)$ differ only by a (nonzero) constant multiple $\varkappa_l$: ${\bf w}^{-}(\lambda_l)= \varkappa_l{\bf w}^{+}(\lambda_l)$.
Hence,
\begin{align}\nonumber
\frac{d}{d\lambda}W({\bf w}^{+}(\lambda),{\bf w}^{-}(\lambda))\Big|_{\lambda=\lambda_l}
&=W_n(\frac 1{\varkappa_l}{\bf w}^{-}(\lambda_l),\frac{d}{d\lambda}{\bf w}^{-}(\lambda_l))
+W_n(\frac{d}{d\lambda}{\bf w}^{+}(\lambda_l),\varkappa_l{\bf w}^{+}(\lambda_l))\\
\label{G4}
&=-\sum\limits_{j\in\Z}{\bf w}^{+}_j(\lambda_l)\cdot{\bf w}^{-}_j(\lambda_l)
=-\varkappa_l\sum\limits_{j\in\Z}{\bf w}^{+}_j(\lambda_l)\cdot{\bf w}^{+}_j(\lambda_l)
\end{align}
by \eqref{G3}.
Thus the poles  of the resolvent at isolated eigenvalues are simple.
Denote $z=\E^{\I\theta}$.  From $2-z-z^{-1}=\lambda^2-m^2$ we obtain
$\displaystyle\frac{d\lambda}{dz}=\displaystyle\frac{1-z^2}{2z^2\lambda}$.
Therefore,
\begin{equation}\label{G5}
\frac{d}{dz}W({\bf w}^{+}(\lambda(z)),{\bf w}^{-}(\lambda(z)))\Big|_{z=z_l}
=\frac{z_l^2-1}{2z_l^2\lambda_l}
\sum\limits_{j\in\Z}{\bf w}^{+}_j(z_l)\cdot{\bf w}^{-}_j(z_l),\quad \lambda_l=\lambda(z_l).
\end{equation}
Now we consider the Jost solution  ${\bf w}^{\pm}(\theta)$, $\theta=\theta(\lambda)$, defined in  \eqref{J2}.
Denote
\[
W^{\pm}(\theta)=W({\bf w}^{\mp}(\theta),{\bf w}^{\pm}(-\theta)).
\]
Recall that the quantities
\begin{equation}\label{TR-def}
T(\theta)= \frac{2\I\sin\theta}{(m+\lambda)W(\theta)},
\quad R^\pm(\theta)= \pm\frac{W^\pm(\theta)}{W(\theta)},\quad  \lambda\in \Gamma_+,
\end{equation}
are known as the transmission and reflection coefficients. For these coefficients
the following scattering relations hold (see \cite{KT1})
\begin{equation}\label{scat-rel}
T(\theta){\bf w}^{\mp}(\theta)= R^{\pm}(\theta){\bf w}^{\pm}(\theta)+{\bf w}^{\pm}(-\theta),
\quad\theta\in[-\pi,\pi].
\end{equation}
Denote $\tilde T(z)=T(\theta(z))$, $\tilde R^{\pm}(z)=R^{\pm}(\theta(z))$. 
Let $F^{\pm}_n$  be the Fourier coefficients of $R^{\pm}$:
\begin{equation}\label{a3}
F^{\pm}_n:=\frac {1}{2\pi \I}\int_{|z|=1}\tilde R^{\pm}(z)z^{\pm n}\frac{dz}{z}.
\end{equation}
Since $|\tilde R^{\pm}(z)|\le 1$ (see \cite{DT, tjac, KT1}), 
Parseval's identity implies
$\sum\limits_{n\in\Z}| F^{\pm}_n|^2=\frac {1}{2\pi \I}\int_{|z|=1}|\tilde R^{\pm}(z)|^2\frac{dz}{z}\le1$.
\\
Hence,  $F^{\pm}\in \ell^2(\Z)$.  Let $\lambda_l\ge 0$, $l=1,...,N$, be  the poles of the resolvent and let
\begin{equation}\label{tiF}
{\mathcal F}^{\pm}_n=F^{\pm}_n +\sum\limits_{l=1}^N \gamma^{\pm}_lz_l^{\pm n}, \quad {\rm where}\quad
\gamma^{\pm}_l=\displaystyle\frac{2\lambda_l}{(m+\lambda_l)\sum\limits_{j\in\Z}{\bf w}^{\pm}_j(z_l)\cdot{\bf w}^{\pm}_j(z_l)}.
\end{equation}
Now  we derive the Gelfand--Levitan--Marchenko equations for ${\mathcal F}^{\pm}$.
\begin{proposition}\label{GLME} (cf. \cite[Equations (10.71), (10.76)]{tjac})
Let $q\in\ell^1_1$. Then \\
{\it i}) For any $j\ge 0$,  the  following equations hold:
\begin{equation}\label{GLM-eq}
\left\{\begin{array}{ll}
a^{+}_{n,j}+{\mathcal F}^{+}_{2n+j}+\sum\limits_{p=0}^{\infty}{\mathcal F}^{+}_{2n+p+j} a^{+}_{n,p}
=\frac{\tilde T(0)(1+a^-_{n,0})}{1-q_n}\delta_{j,0},\\
b^{+}_{n,j-1}+{\mathcal F}^{+}_{2n+j}-{\mathcal F}^{+}_{2n+j-1}+\!\sum\limits_{p=-1}^{\infty}{\mathcal F}^{+}_{2n+p+j-1}b^{+}_{n,p}\\
\quad=(1-\!q_n)\Big[\tilde T(0)(1+b^{-}_{n,0})\delta_{j-1,0}+\big(\tilde T'(0)(b^{-}_{n,0}-1)+\tilde T(0)(b^{-}_{n,-1}+\!1)\big)\delta_{j-1,-1}\Big].
\end{array}\right.
\end{equation}
{\it ii})  For any $j\le 0$,  the following equations hold:
 \begin{equation}\label{GLM-2} 
\left\{\begin{array}{ll}
a^{-}_{n,j}+{\mathcal F}^{-}_{2n+j}+ \sum\limits_{p=0}^{-\infty}{\mathcal F}^{-}_{2n+p+j} a^{-}_{n,p}=(1-q_n)[\tilde T(0)(1+a_{n,0}^+)\delta_{j,0}],\\
b^{-}_{n,j}+{\mathcal F}^{-}_{2n+j}+\sum\limits_{p=0}^{-\infty}{\mathcal F}^{-}_{2n+p+j}b^{-}_{n,p}=\displaystyle\frac{\tilde T(0)(b^{+}_{n,0}-1)}{1-q_n}\delta_{j,0}.
\end{array}\right.
\end{equation}
{\it ii}i) The following estimates hold
\begin{equation}\label{tF-est}
|{\mathcal F}_{n}^{\pm}|\le M^{\pm}_n\sum\limits_{p=[\frac n2]}^{\pm\infty}(|q_p|+|\tilde q_p|),
\end{equation}
where $M^{\pm}_n$  are terms of order zero as $n\to\pm\infty$.
\end{proposition}
\begin{proof}
{\it i}). Consider  (\ref{scat-rel}) with  upper signs:
\begin{equation}\label{al}
\left\{\begin{array}{cc}
\tilde T(z)\tilde u^{-}(z)=\tilde u^{+}(z^{-1})+\tilde R^{+}(z)\tilde u^{+}(z),\\
\tilde T(z)\tilde w^{-}(z)=\tilde w^{+}(z^{-1})+\tilde R^{+}(z)\tilde w^{+}(z),
\end{array} \right.
\end{equation}
where $\tilde u^{\pm}(z)=u^{\pm}(\theta(z))$,  $\tilde w^{\pm}(z):=(m+\lambda(z))v^{\pm}(\theta(z))$.
We multiply the first equation by $(2\pi i)^{-1}z^{n+j}$, $j=0,1...$, and integrate around the unit circle.
Using \eqref{B}, we first evaluate the right hand side:
\begin{equation}\label{a2}
\frac {1}{2\pi \I}\int\limits_{|z|=1}\tilde u^{+}_n(z^{-1})z^{n+j}\frac{dz}{z}=a^+_{n,j},\qquad 
\frac {1}{2\pi \I}\int\limits_{|z|=1}\tilde R^{+}(z)\tilde u^{+}_n(z)z^{n+j}\frac{dz}{z}=
\sum\limits_{p=0}^{\infty} F^{+}_{2n+p+j}a^{+}_{n,p}.
\end{equation}
Next we evaluate the left hand side.
From  \eqref{G5}  and \eqref{TR-def}  it follows that
\begin{align*}
{\rm res}_{z_l}\tilde T(z)\tilde u^{-}_n(z)z^{n+j-1}&=
{\rm res}_{z_l}\frac{(z^2-1)\tilde u^{-}_n(z)z^{n+j-2}}{(m+\lambda)W({\tilde {\bf w}}^{+}(z),{\tilde {\bf w}}^{-}(z))}\\
&=\frac{2\lambda_l\tilde u^{-}_n(z_l)z_l^{n+j}}{(m+\lambda_l)
\sum\limits_{j\in\Z}{\bf w}^{+}_j(z_l)\cdot{\bf w}^{-}_j(z_l)}=\gamma^{+}_l\tilde u^{+}_n(z_l)z_l^{n+j}.
\end{align*}
Using  \eqref{B} and the residue theorem
(take a contour inside the unit disk enclosing all poles and let this contour
approach the unit circle), we obtain
\begin{align}\nonumber
&\frac {1}{2\pi \I}\int_{|z|=1}\tilde T(z)\tilde u^{-}_n(z)z^{n+j}\frac{dz}{z}=
-\sum\limits_{l=1}^N \gamma^{+}_l\tilde u^{+}_n(z_l)z_l^{n+j}+\tilde T(0)(\tilde h^{-}_{n}(0))_1\delta_{j,0}\\
\label{a4}
&=-\sum\limits_{p=0}^{\infty}a_{n,p}^{+}\sum\limits_{l=1}^N \gamma^{+}_lz_l^{2n+p+j}
+\frac{\tilde T(0)(1+a^-_{n,0})}{1-q_n}\delta_{j,0},
\end{align}
where $\tilde T(0)<\infty$ (see   Appendix~\ref{App}), and  $(\tilde h^{-}_{n})_1$ is the first component of the vector $\tilde h^{-1}_n$.
Substituting \eqref{a2} and \eqref{a4} into the first equation of \eqref{al}, we obtain the first equation of \eqref{GLM-eq}.

Now consider the second equation of \eqref{al}. Similarly to \eqref{a2}--\eqref{a4}, we obtain for $j=-1,0,1,...$
\begin{align*}
\frac {1}{2\pi \I}\int_{|z|=1}\tilde w^{+}_n(z^{-1})z^{n+j}\frac{dz}{z}&=\frac{b_{n,j}^{+}}{1-q_n},\\
\frac {1}{2\pi \I}\int_{|z|=1}\tilde R^{+}(z)\tilde w^{+}_n(z)z^{n+j}\frac{dz}{z}&=
\sum\limits_{p=-1}^{\infty} F^{+}_{2n+p+j} \frac{b^{+}_{n,p}}{1-q_n},\\
\frac {1}{2\pi \I}\int_{|z|=1}\tilde T(z)\tilde w^{-}_n(z)z^{n+j}\frac{dz}{z}
&=-\!\sum\limits_{p=-1}^{\infty} b_{n,p}^{+}\sum\limits_{l=1}^N \gamma^{+}_lz_l^{2n+p+j}+\tilde T(0)(1+b^{-}_{n,0})\delta_{j,0}\\
&+\big(\tilde T'(0)(b^{-}_{n,0}-1)+\tilde T(0)(b^{-}_{n,-1}+1)\big)\delta_{j,-1},
\end{align*}
where $\tilde T(0),\tilde T'(0)<\infty$ (see   Appendix~\ref{App}). Then the second equation of \eqref{GLM-eq} follows.
\smallskip\\
{\it ii}) Equation  (\ref{scat-rel})   with lower signs reads
\[
\left\{\begin{array}{cc}
\tilde T(z)\tilde u^{+}(z)=\tilde u^{-}(z^{-1})+\tilde R^{-}(z)\tilde u^{-}(z),\\
\tilde T(z)\tilde w^{+}(z)=\tilde w^{-}(z^{-1})+\tilde R^{-}(z)\tilde w^{-}(z).
\end{array} \right.
\]
Multiplying by $(2\pi\I)^{-1}z^{n+j}$, $j=0,-1,-2\dots$, and integrating around the unit circle, we obtain \eqref{GLM-2}.
\smallskip\\
{\it iii}) Note that $|a_{n,p}^{\pm}|< 1$ for sufficiently large $\pm n$ by \eqref{est3}.
Hence,  equations \eqref{GLM-eq}--\eqref{GLM-2}  together with the estimates \eqref {est3}    imply
\begin{align*}
|{\mathcal F}^{\pm}_{2n+j}|&\le |a^{\pm}_{n,j}|+\sum\limits_{p=0}^{\pm\infty}|{\mathcal F}^{\pm}_{2n+p+j} a^{\pm}_{n,p}|\\
&\le C_n^{\pm}\Big(\cQ^{\pm}\big(n\pm 1
+[\frac j2]\big)+\sum\limits_{p=0}^{\pm\infty}|{\mathcal F}^{+}_{2n+p+j}|\cQ^{\pm}\big(n\pm 1+[\frac p2]\big)\Big), \quad \pm j\ge 1,
\end{align*}
where $\cQ^{\pm}(n)=\sum\limits_{l=n}^{\pm\infty}( |q_l|+|\tilde q_l|)$. Then \eqref{tF-est} follows by arguments 
from  \cite{Gus} and \cite[Section 10.3]{tjac}. 
\end{proof}
\section{The Wiener algebra}
\label{Winer}

Recall that the Wiener algebra is the set of all integrable functions
whose Fourier coefficients are integrable:
\[
\mathcal{A} = \Big\{ f(\theta) = \sum_{m\in\Z} \hat{f}_m \E^{\I m \theta}\  \Big|\, \|\hat{f}\|_{\ell^1} < \infty \Big\}.
\]
We set
\[
\Vert f\Vert_{\mathcal{A}}=\Vert \hat f\Vert_{\ell^1},\qquad  \Vert (f_1,f_2)\Vert_{\mathcal{A}}=\Vert (\hat f_1,\hat f_2)\Vert_{\bl^1}.
\]
Since $\lambda=\lambda(\theta)=\sqrt{2-2\cos\theta+m^2}\in C^{\infty}([-\pi,\pi])$ 
we have $\lambda+m$, $\frac 1{\lambda+m}\in{\mathcal{A}}$. Hence, the
representation \eqref{B} and the estimates \eqref{est3} imply that
\begin{equation} \label{alg1}
{\bf h}^\pm_n(\theta), {\bf w}^\pm_n(\theta)\in \mathcal A \quad \mbox{if} \quad q\in\ell^1_1.
\end{equation}
Consequently, the Wronskians $W(\theta)$ and $W^\pm(\theta)$  also belong to  $\mathcal A$.
\begin{theorem}\label{thm:scat} 
If $q\in\ell^1_1$, then  $T(\theta)$, $R^\pm(\theta)\in \mathcal{A}$.
\end{theorem}
\begin{proof} 
Due to Lemma \ref{W}, $W(\theta)$ can vanish only at the edges of continuous spectra,
i.e. when $\theta=0,\pi$, which correspond to the resonant cases.
(We identify the points $\pi$ and $-\pi$ and consider the Jost solutions as functions on the unit circle.)
In the  case $W(0)W(\pi)\neq 0$, $W(\theta)^{-1} \in \mathcal{A}$ by Wiener's lemma. Therefore, $T(\theta), R^\pm (\theta)\in \mathcal A$.
It remains to consider the case   $W(0)W(\pi)=0$. 
\begin{lemma}\label{EMT} 
Let  $W(0)=0$. Then  the following representations hold
\[
(m+\lambda)W(\theta) =(1-\E^{\I\theta})\Phi(\theta),\quad (m+\lambda)W^\pm(\theta)=(1 -\E^{\I\theta})\Phi^\pm(\theta),\quad\lambda=\lambda(\theta),
\]
where $\Phi(\theta), \Phi^\pm(\theta)\in\mathcal A$.  
Moreover, if $W(\pi)=0$ then $\Phi(\theta)\neq 0$ for $\theta\in (-\pi,\pi)$,  and if
$W(\pi)\neq 0$ then $\Phi(\theta)\neq 0$ for $\theta\in[-\pi,\pi]$.
\end{lemma}
\begin{proof}
Denote $w^{\pm}_n(\theta):=(m+\lambda)v^{\pm}_n(\theta)$. Since
\begin{equation}\label{mW}
W(0)= u^+_0(0) \frac{w^-_{1}(0)}{2m} - \frac{w^+_{1}(0)}{2m} u^-_0(0)=0,
\end{equation}
we have two possible combinations (because the solutions ${\bf w}^\pm_n(0)$ cannot vanish at two
consecutive points): 
\[
(a):\quad u^+_0(0)u^-_0(0)\neq 0  \quad\quad{\rm and}\quad\quad
 (b): \quad w^+_1(0)w^-_1(0)\neq 0.
 \]
We will only consider the case (a).  The case (b) is treated similarly.
By \eqref{W-def} and \eqref{mW} we get
\begin{equation}\label{WV-rep}
(m+\lambda)W(\theta) =u_0^+(\theta)u_0^-(\theta)
\left(\frac{V^+(\theta)}{u_0^+(0)u_0^+(\theta)}-\frac{V^-(\theta)}{u_0^-(0)u_0^-(\theta)}\right),
\end{equation}
where $V^\pm(\theta):=u_0^\pm(\theta)w_1^\pm(0) - u_0^\pm(0)w_1^\pm(\theta)$.
\smallskip\\
{\it Step i)} Let us prove that
\begin{equation}\label{alg6}
V^\pm(\theta)=(1 - \E^{\I\theta})\Psi^\pm(\theta),\quad
V^\pm(\theta)=(1 +\E^{\I\theta})\tilde\Psi^\pm(\theta)
\end{equation}
with
\begin{equation}\label{alg8}
\Psi^\pm(\theta), ~~\tilde\Psi^\pm(\theta)\in\mathcal A.
\end{equation}
We consider the case "+" and the first equality in \eqref{alg6} only.
Representation \eqref{B} implies
\begin{equation}\label{B1}
u_n^+(\theta)=\sum_{k=n}^{\infty}\tilde a_{n,k}^+z^k,\quad \quad
w_n^+(\theta)= (m+\lambda)v_n^+(\theta)=\sum_{k=n-1}^{\infty}\tilde b_{n,k}^+z^k,\quad  z=\E^{\I\theta},
\end{equation}
where 
\begin{equation}\label{ta}
\tilde a_{n,k}^+=\delta_{n,k}+a_{n,k-n}^+,\qquad \tilde b^{+}_{n,k}=\frac{\delta_{k,-1}-\delta_{k,0}+b^{+}_{n,k-n}}{1-q_n}.
\end{equation}
We will use summation by parts, i.e., the following identity,
\begin{equation}\label{sump}
\sum\limits_{k=s}^{\infty}(f(k)-f(k+1))g(k)=
\sum\limits_{k=s}^{\infty}f(k)(g(k)-g(k-1))+f(s)g(s-1),
\end{equation}
which is valid for all $f\in\ell^1(\Z_+)$, $g\in\ell^{\infty}(\Z_+)$ or vice versa.
Introduce
\begin{equation}\label{B2}
a_n(s)=\sum\limits_{k=s}^{\infty}{\tilde a}_{n,k}^+,\quad\quad
b_n(s)=\sum\limits_{k=s}^{\infty}{\tilde b}_{n,k}^+,
\end{equation}
which are well defined due to \eqref{est3}. Note, that
$a_{n}(n)=u_n^+(0)$ and  $b_{n}(n-1)=w_n^+(0)$.
Applying \eqref{sump} to \eqref{B1} and using \eqref{B2},  we obtain
\begin{align*}
u_n^+(\theta)&=\sum_{k=n}^{\infty}(a_n(k)-a_n(k+1))z^k=\sum_{k=n}^{\infty}a_n(k)z^k(1-z^{-1})+u_n^+(0)z^{n-1},\\
w_n^+(\theta)&=\sum_{k=n-1}^{\infty}(b_n(k)-b_n(k+1))z^k=\sum_{k=n-1}^{\infty}b_n(k)z^k(1-z^{-1})+w_n^+(0)z^{n-2}.
\end{align*}
Abbreviate $\zeta(z)=(z-1)/z$, then
\begin{equation}\label{B4}
u_0^+(\theta)=\zeta(z)\sum_{k=1}^{\infty}a_0(k)z^k+u_0^+(0),\qquad
w_1^+(\theta)=\zeta(z)\sum_{k=1}^{\infty}b_1(k)z^k+w_1^+(0).
\end{equation}
Multiplying the first equation of \eqref{B4} by $w_1^+(0)$ and the second equation  by $u_0^+(0)$, their difference is equal to
\begin{equation}\label{B5}
V^+(\theta)=u_0^+(\theta)w_1^+(0)-w_1^+(\theta)u_0^+(0)
=(1-\E^{\I\theta})\sum\limits_{k=0}^{\infty}g(k)\E^{\I k\theta},
\end{equation}
where
\begin{equation}\label{B7}
g(k)=a_0(k)w_1^+(0)-b_1(k)u_0^+(0).
\end{equation}
Note that by \eqref{est3} and \eqref{B2},  we have $g(\cdot)\in \ell^{\infty}(\Z_+)$. It remains to show that
\begin{equation}\label{B9}
g(\cdot)\in \ell^{1}(\Z_+).
\end{equation}
The Gelfand--Levitan--Marchenko equations \eqref{GLM-eq} imply
\[
\tilde a_{0,j}+\sum\limits_{p=0}^{\infty}{\mathcal F}_{p+j}\tilde a_{0,p}=0,\quad\quad
\tilde b_{1,j}+\sum\limits_{p=0}^{\infty}{\mathcal F}_{p+j}\tilde b_{1,p}=0,\quad j\ge 2.
\]
Summing both equalities from $s\ge 2$ to $\infty$  gives
\begin{align*}
&a_{0}(s)+\sum\limits_{j=s}^{+\infty}\sum\limits_{p=0}^{+\infty}{\mathcal F}_{p+j}[a_{0}(p)-a_{0}(p+1)]=0,\\
&b_{1}(s)+\sum\limits_{j=s}^{+\infty}\sum\limits_{p=0}^{+\infty}{\mathcal F}_{p+j}[b_{1}(p)- b_{1}(p+1)]=0.
\end{align*}
Applying \eqref{sump}, we obtain
\begin{align*}
& a_{0}(s)+\sum\limits_{j=s}^{+\infty}\Big(\sum\limits_{p=0}^{+\infty}({\mathcal F}_{p+j}-{\mathcal F}_{p+j-1})
 a_{0}(p)+a_{0}(0){\mathcal F}_{j-1}\Big)=0,
\\
&b_{1}(s)+\sum\limits_{j=s}^{+\infty}\Big(\sum\limits_{p=0}^{+\infty}({\mathcal F}_{p+j}-{\mathcal F}_{p+j-1})
b_{1}(p)+b_{1}(0){\mathcal F}_{j-1}\Big)=0.
\end{align*}
Hence,  
\begin{align}\nonumber
&a_{0}(s)+u_0^+(0)\sum\limits_{j=s}^{+\infty}{\mathcal F}_{j-1}-\sum\limits_{p=0}^{+\infty}a_{0}(p){\mathcal F}_{p+s-1}=0,\\
\nonumber
&b_{1}(s)+w_1^+(0)\sum\limits_{j=s}^{+\infty}{\mathcal F}_{j-1}-\sum\limits_{p=0}^{+\infty}b_{1}(p){\mathcal F}_{p+s-1}=0
\end{align}
by \eqref{B2}.
We multiply the first equation  by $w_1^+(0)$, the second by $u_0^+(0)$,
 subtract the second equation from the first, and use \eqref{B7} to arrive at
\begin{equation}\label{alg5}
g(s)-\sum\limits_{p=0}^{+\infty}g(p){\mathcal F}_{p+s-1}=0.
\end{equation}
Any bounded solution to \eqref{alg5} with a kernel satisfying \eqref{tF-est} belongs to $\ell^1(\Z_+)$ as proved 
in \cite{Mar}. Hence, \eqref{B9} follows.
\smallskip\\
{\it Step ii)} Substituting \eqref{alg6} into  \eqref{WV-rep}, we obtain
\[
(m+\lambda)W(\theta) 
=(1-\E^{\I\theta})\left(\frac{u^-_0(\theta)}{u^+_0(0)}\Psi^+(\theta)
- \frac{u^+_0(\theta)}{u^-_0(0)}\Psi^-(\theta)\right)=(1-\E^{\I\theta})\Phi(\theta),
\]
where $\Phi(\theta)\in\mathcal A$ by \eqref{alg8} and \eqref{alg1}. 
We observe that if $W(\pi)=0$ then $\Phi(\theta)\neq 0$ for $\theta\in (-\pi,\pi)$,  
and if $W(\pi)\neq 0$ then $\Phi(\theta)\neq 0$ for $\theta\in[-\pi,\pi]$.

Since  $W(0)=0$ implies $W^\pm(0)=0$ then  we can also get similarly
$(m+\lambda)W^\pm(\theta)=(1 -\E^{\I\theta})\Phi^\pm(\theta)$ with $\Phi^\pm(\theta)\in\mathcal A$.
\end{proof}
Analogously, $W(\pi)=0$ implies 
\[
(m+\lambda)W(\theta)= (1+\E^{\I\theta}) \tilde{\Phi}(\theta),\quad
(m+\lambda)W^\pm(\theta)= (1+\E^{\I\theta}) \tilde{\Phi}^\pm(\theta)
\]
 with 
$\tilde{\Phi},\tilde{\Phi}^\pm\in \mathcal A$ and $\tilde\Phi(\theta)\neq 0$ for $\theta\in[-\pi,\pi]$ if $W(0)\neq 0$. 
Thus, if $W$ vanishes at only one endpoint, this finishes the proof. If $W$ vanishes at both endpoints, 
we can use a smooth cut-off function to combine both representations into
$(m+\lambda)W(\theta)= (1-\E^{2\I\theta}) \breve{\Phi}(\theta)$ (respectively,
$(m+\lambda)W^\pm(\theta)= (1-\E^{2\I\theta}) \breve{\Phi}^\pm(\theta)$) with $\breve{\Phi},\breve{\Phi}^\pm\in \mathcal A$ 
and $\breve{\Phi}(\theta)\neq 0$ for $\theta\in[-\pi,\pi]$.
\end{proof}
\section{The case $\re\lambda \le 0$.}
\label{lav-sec}
In the case   $\lambda\in\Xi_{-}=\{\lambda\in\C\setminus\overline\Gamma_-,~\re\lambda\le 0\}$, where  $\Gamma_-=(-\sqrt{4+m^2},- m)$,
the Jost solutions of (\ref{D1}) are defined according the boundary conditions
\begin{equation}\label{tJ2}
\check {\bf w}_n^{\pm}(\theta)=
\left(\begin{array}{ll}
\check u_n^{\pm}(\theta)\\
\check v_n^{\pm}(\theta)
\end{array}\right)
\to
\left(\begin{array}{cc}
\check\alpha_{\pm}(\theta)\\
1
\end{array}\right)\E^{\pm\I\theta n},\quad n\to\pm\infty, \quad {\rm where}~~\check\alpha_{\pm}(\theta)=\frac{\E^{\pm\I\theta}-1}{\lambda-m}.
\end{equation}
Obviously, Lemmas \ref{Jost-sol} and \ref{W} hold also for
$\check{\bf h}_n^{\pm}(\theta)=\check{\bf w}_n^{\pm}(\theta)\E^{\mp \I\theta n}$ and $\check W(\theta)=W(\check w^+(\theta), \check w^-(\theta))$.
Furthermore, for any $\lambda\in  \Xi_-$, the matrix elements of the resolvent  $\cR(\lambda):=({\mathcal D}-\lambda)^{-1}:\bl^2 \to \bl^2$ 
are given by
\[
[\cR(\lambda)]_{n,k}=\frac{1}{\check W(\theta(\lambda))}
\left\{\begin{array}{cc}
\check{\bf w}_n^+(\theta(\lambda))\otimes\check{\bf w}_k^-(\theta(\lambda)),\quad k\le n,\\
\check{\bf w}_n^-(\theta(\lambda))\otimes\check{\bf w}_k^+(\theta(\lambda)),\quad k\ge n.
\end{array}\right.
\]
For  $\lambda\in\Gamma_-$, the following convergence
\[
    \cR(\lambda\pm \I\varepsilon)\to \cR(\lambda\pm \I0),\quad\varepsilon\to 0+
\]
holds in ${\mathcal L}(\bl^2_\sigma,\bl^2_{-\sigma})$ with $\sigma>1/2$. Here
\[
[\cR(\lambda\pm \I0)]_{n,k} = \frac{1}{\check W(\theta_{\pm})}\left\{ \begin{array}{cc}
\check{\bf w}_n^+(\theta_{\pm})\otimes \check{\bf w}_k^-(\theta_{\pm}) \;\; {\rm for} \;\; n \le k, \\\\
\check{\bf w}_k^+(\theta_{\pm})\otimes \check{\bf w}_n^-(\theta_{\pm}) \;\; {\rm for} \;\; n\ge k. 
\end{array} \right.
\]
Calculations similar to calculations in the Proposition \ref{Fh} lead to the representations
\begin{equation}\label{B-}
A_n^{\pm}\check {\bf h}^\pm_n(\theta) = \left(\begin{array}{ll}
\check\alpha_{\mp}(\theta)\\
1
\end{array}\right)
+\sum_{k=\mp 1}^{\pm\infty} \left(\begin{array}{cc}
\frac{\check a^\pm_{n,k}}{\lambda-m}\\ \check b^\pm_{n,k}\end{array}\right)
\E^{\pm \I k \theta},
\end{equation}
where 
\begin{equation}\label{est3-}
|\check a^\pm_{n,k}|, | \check b^\pm_{n,k}|
\le  \check C^\pm_n  \sum_{l=n+1+[k/2]}^{\pm\infty}( |q_l|+\frac{ |q_l|}{|1-q_l|}),
\end{equation}
and
\begin{equation}\label{estC-}
\quad \check C^\pm_n\le \check C^{\pm},\quad \mbox{if}\  \pm n\geq 0.
\end{equation}
Denote 
\[
\check W^{\pm}(\theta)=W(\check{\bf w}^{\mp}(\theta),\check{\bf w}^{\pm}(-\theta)),\quad 
\check T(\theta)= \frac{2\I\sin\theta}{(\lambda-m)\check W(\theta)},
\quad \check R^\pm(\theta)= \pm\frac{\check W^\pm(\theta)}{\check W(\theta)},\quad  \lambda\in \Gamma_-.
\]
Finally,  $\check T(\theta)$, $\check R^\pm(\theta)\in \mathcal{A}$ for $q\in\ell^1_1$.
This follows  similarly  to  Theorem \ref{thm:scat}  using  the corresponding Gelfand--Levitan--Marchenko  equations and  estimates of type \eqref{tF-est} for its coefficients.
\section{Dispersive decay}
\label{dd-sec}
We  will use the following variant of the van der Corput lemma. 
\begin{lemma}\label{lem:vC} (see \cite{EKT})
Consider the oscillatory integral
\begin{equation}
I(t) = \int_a^b \E^{\I t \phi(\theta)} f(\theta) d\theta, \qquad -\pi \le a < b \le \pi,
\end{equation}
where $\phi(\theta)$ is a real-valued smooth function and $f\in\mathcal{A}$.
If $|\phi^{(k)}(\theta)| >0$ for some $k\ge 2$ and for any $\theta\in [a,b]$  then
\[
|I(t)| \le C_k \big(t\min_{[a,b]}|\phi^{(k)}(\theta)|\big)^{-1/k}\|f\|_{\mathcal{A}}, \quad t\ge 1,
\]
where $C_k$ is a universal constant.
\end{lemma}
\begin {theorem}\label{end1}
Let $q\in\ell^1_1$.
Then the asymptotics \eqref{fullp} and \eqref{as1-new} hold. 
\end{theorem}
\begin{proof}
We split $\E^{-\I t\D}P_{c}$ as follows
\begin{equation}  \label{sp-rep-sol}  
\E^{-\I t\D}P_{c}= \E^{-\I t\D}P_{c}^++ \E^{-\I t\D}P_{c}^-,\qquad \E^{-\I t\D}P_{c}^{\pm}
   =\frac 1{2\pi \I}\int\limits_{\Gamma_{\pm}}
   \E^{-\I t\lambda}( \cR(\lambda+\I 0)- \cR(\lambda-\I 0))\,d\lambda,
 \end{equation}
 and prove  asymptotics \eqref{fullp} and \eqref{as1-new} for  the first summand only.
 Using \eqref{RJ1-rep}, we express 
the matrix elements of $\E^{-\I t\D} P_{c}^+$  in terms of the Jost solutions, (cf. \cite[Formula 6.5]{KT1}):
\begin{equation}\label{sp-rep-sol1}
\left[ \E^{-\I t\D}P_{c}^+ \right]_{n,k}=
\frac{1}{2\pi \I} \int_{-\pi}^{\pi} \frac{\E^{-\I t\sqrt{2-2\cos\theta+m^2}}}{\sqrt{2-2\cos\theta+m^2}}\,
\frac{{\bf w}_k^+(\theta)\otimes {\bf w}_n^-(\theta)}{W(\theta)}  \sin\theta\, d\theta,\quad n\le k,
\end{equation}
and by symmetry 
$\left[ \E^{-\I t\D}P_{c} \right]_{n,k}= \left[ \E^{-\I t\D}P_{c} \right]_{k,n}$ for $n\ge k$.
\smallskip\\
{\it Step i)}
We first prove asymptotics of type (\ref{fullp}) for  $\E^{-\I t\D}P_{c}^+$.  It is enough to check that
\begin{equation}
\label{eq:m21}
\sup\limits_{n,k}|\left[ \E^{-\I t\D}P_{c}^+\right]_{n,k}| \le C t^{-1/3},\quad t\to\infty.
\end{equation}
We suppose $n\le k$ for notational simplicity. Using \eqref{TR-def}, we rewrite \eqref{sp-rep-sol1} as
\begin{equation}\label{EDP-rep}
\left[ \E^{-\I t\D}P_{c}^+ \right]_{n,k} =  -\frac{1}{4\pi} \int_{-\pi}^{\pi} (m+\lambda)
\frac{\E^{-\I t[g(\theta)-\frac{k-n}t\theta]}}{g(\theta)}\,T(\theta){\bf h}_k^+(\theta)\otimes{\bf h}_n^-(\theta) d\theta,
\end{equation}
where $g(\theta):=\sqrt{2-2\cos\theta+m^2}$.
We also apply  the scattering relations  \eqref{scat-rel}
to get  the representations
\begin{equation}\label{Y_nk}
T(\theta){\bf h}_k^+(\theta)\otimes{\bf h}_n^-(\theta)=\left\{\!\!\!\begin{array}{ll}
R^-(\theta){\bf h}_n^-(\theta)\otimes{\bf h}_k^-(\theta)\E^{-2\I k\theta} + {\bf h}_n^-(\theta)\otimes{\bf h}_k^-(-\theta), & n\leq k\leq 0,\\[1.5mm]                             
R^+(\theta){\bf h}_k^+(\theta)\otimes{\bf h}_n^+(\theta)\E^{2\I n\theta} + {\bf h}_k^+(\theta)\otimes{\bf h}_n^+(-\theta), & 0\leq n\leq k.
\end{array}\right.
\end{equation}
Using  the facts
\begin{equation}\label{nk}
k-n-2k=-(k+n)=|k+n|,\quad n\le k\le 0,
\end{equation}
\begin{equation}\label{nk1}
k-n+2n=k+n=|k+n|,\quad 0\le n\le k,
\end{equation}
and abbreviating $v:=\frac{k-n}t\ge 0$, $\tilde v:=\frac{|n+k|}t\ge 0$,
we  finally rewrite \eqref{EDP-rep}  as
\begin{equation}\label{EDP-rep1}
\left[ \E^{-\I t\D}P_{c}^+ \right]_{n,k}=-\frac{1}{4\pi} \int_{-\pi}^{\pi}
\frac{\E^{-\I t\Phi_v(\theta)}}{g(\theta)}\,Y_{n,k}^{1}(\theta)d\theta
-\frac{1}{4\pi} \int_{-\pi}^{\pi}
\frac{\E^{-\I t\tilde\Phi_v(\theta)}}{g(\theta)}\,Y_{n,k}^{2}(\theta)d\theta,
\end{equation}
where 
\begin{align}\nonumber
Y_{n,k}^1(\theta)&=\frac{m+\lambda}{g(\theta)}\left\{\begin{array}{ll}
 T(\theta){\bf h}_k^+(\theta)\otimes{\bf h}_n^-(\theta),~~n\leq 0\leq k,\\
{\bf h}_n^-(\theta)\otimes {\bf h}_k^-(-\theta),~~n\leq k\leq 0,\\
{\bf h}_k^+(\theta)\otimes {\bf h}_n^+(-\theta),~~0\leq n\leq k,
\end{array}\right.\\
\nonumber
\\
\nonumber
Y_{n,k}^2(\theta)&=\frac{m+\lambda}{g(\theta)} \left\{\begin{array}{ll} 0,~~n\le 0\le k,\\
R^-(\theta){\bf h}_n^-(\theta)\otimes{\bf h}_k^-(\theta),~~n\leq k\leq 0,\\
R^+(\theta){\bf h}_n^+(\theta)\otimes{\bf h}_k^+(\theta),~~0\leq n\leq k,
\end{array}\right.
\end{align} 
and
\begin{equation}\label{phi}
\Phi_v(\theta)=g(\theta)-v\theta,\quad\tilde\Phi_v(\theta)=g(\theta)-\tilde v\theta.
\end{equation}
 We observe that the matrix functions 
$Y_{n,k}^j(\theta)$ belong to $\mathcal A$, and
the $\ell_1$ -norm of its Fourier coefficients 
can be estimated by a value, which does not depend on $n$ and $k$.
Indeed, \eqref{est3}--\eqref{estC} imply that
\[
\sup_{\pm n >0}\sum_{k=0}^{\pm\infty }|a_{n,k}^\pm|
+\sup_{\pm n >0}\sum_{k=0}^{\pm\infty }|b_{n,k}^\pm|\le  C<\infty.
\]
Hence,
\begin{equation}\label{est10}
\|\hat {\bf h}^\pm_n\|_{\bl^1}\leq C, \quad\mbox{for}\quad\pm n\ge 0,
\end{equation}
and Theorem \ref{thm:scat} implies that
\begin{equation}\label{Ynk-est}
\| Y_{n,k}^{j}(\cdot)\|_{\mathcal A}\leq C. 
\end{equation}
Denote  $\varkappa:=(2+m^2-\sqrt{4m^2+m^4})/2$, $0<\varkappa<1$.
It is easy to check that if $v\not= \sqrt\varkappa$ then the phase function $\Phi_v(\theta)$
has at most two non-degenerate stationary points.
In the case  $v=\sqrt\varkappa$ there exists a unique degenerate stationary point $\theta_0=\arccos\varkappa$,
$0<\theta_0<\pi/2$,
such that  $\Phi'''(\theta_0)=\sqrt{\varkappa}\not =0$. 
The function $\tilde\Phi_v(\theta)$ has the same properties.

We split the domain of integration in \eqref{EDP-rep1} into regions where
either the second or third derivative of the phases is nonzero and apply Lemma~\ref{lem:vC} together with 
\eqref{Ynk-est} to obtain the asymptotics \eqref{eq:m21}.
\smallskip\\
{\it Step ii)}
Now we  prove the asymptotics of type (\ref{as1-new}) for  $\E^{-\I t\D}P_{c}^+$. Denote $G=\max\limits_{\theta\in [-\pi,\pi]}|g^{\prime\prime\prime}(\theta)|$ and set
\begin{equation}\label{Jpm}
{\bf J}_{\pm}=\{\theta:\:|\theta\mp\theta_0|\le \nu|\theta_0|\},\quad
{\bf J}=[-\pi, \pi]\setminus ({\bf J}_+\cup {\bf J}_-),
\end{equation}
where $\nu=\min\{\frac 12, \sqrt{\frac{2\sqrt\varkappa}{3G\theta_0^2}}\,\}$.
We represent $\E^{-\I t\D}P_{c}^+$ as the sum
\begin{equation}\label{KK-rep}
\E^{-\I t\D}P_{c}^+ ={\mathcal K}^{\pm}(t)+{\mathcal K}(t),
\end{equation}
where
\begin{align*}
[{\mathcal K}^\pm(t)]_{n,k}&=-\frac{1}{4\pi} \int_{{\bf J}_{\pm}}
\Big[\E^{-\I t\Phi_v(\theta)} Y_{n,k}^1(\theta) +
\E^{-\I \tilde\Phi_v(\theta)} Y_{n,k}^2(\theta)\Big] \frac{d\theta}{g(\theta)},\\
[{\mathcal K}(t)]_{n,k}&=-\frac{1}{4\pi} \int_{{\bf J}}
\Big[\E^{-\I t\Phi_v(\theta)} Y_{n,k}^1(\theta)+\E^{-\I \tilde\Phi_v(\theta)} Y_{n,k}^2(\theta)\Big] 
\frac{d\theta}{g(\theta)}.
\end{align*}
The van der Corput Lemma \ref{lem:vC} with $k=2$ together with \eqref{Ynk-est}  imply
\[
\sup_{n,k\in\Z}|[{\mathcal K}(t)]_{n,k}|\le C t^{-1/2},\quad t\ge 1.
\]
Hence, the asymptotics of type  (\ref{as1-new}) for $ {\mathcal K}(t)$ follow. It remains to 
prove the asymptotics  for  ${\mathcal K}^{\pm}(t)$.
Since $W(\theta)\neq 0$ for $\theta\in {\bf J}_{\pm}$, then 
\begin{equation}\label{dTR-est}
|\frac{d}{d\theta}T(\theta))|,~|\frac{d}{d\theta}R^\pm(\theta)|\le C,\quad \theta\in {\bf J}_{\pm}
\end{equation}
by Proposition \ref{Jost-sol} (i). Hence, \eqref{dh-est} and \eqref{dTR-est} imply
\begin{equation}\label{phi-c}
|Y_{n,k}|+|\frac{d}{d\theta} Y_{n,k}|\le C,~~ \theta\in {\bf J}_{\pm},~~j=1,2.
\end{equation}
Moreover,
\[
|\Phi_v'(-\theta_0\pm\theta)|=|\frac{-\sin(\theta_0\mp\theta)}{\sqrt{2-2\cos(\theta_0\pm\theta)+m^2}}-v|
\ge \frac{\sin(\theta_0\mp\theta)}{\sqrt{4+m^2}}\ge  \frac{\sin(\theta_0/2)}{\sqrt{4+m^2}}>C>0,\quad \theta\in J_-.
\]
Therefore, applying integration by parts, we obtain
\[
\sup_{n,k\in\Z}|[{\mathcal K}^-(t)]_{n,k}|\le C t^{-1},\quad t\ge 1,
\]
which implies  asymptotics of type  (\ref{as1-new}) for $ {\mathcal K}^-(t)$.
Finally, the asymptotics for  ${\mathcal K}^+(t)$ follow by  Lemma 6.3 from  \cite {KT1} (with $p=0$) due to \eqref{phi-c}.
\end{proof}
\section{The non-resonant case}
\label{dd1-sec}
\begin {theorem}\label{t-newn}
Let $q\in\ell^1_2$. Then in the non-resonant case the asymptotics \eqref{as-new}  hold.
\end{theorem}
\begin{proof}
It suffices to show that
\begin{equation}\label{HP-n1}
|\left[ \E^{-\I t \D}P_c \right]_{n,k}|\le C (1+|n|)(1+|k|)t^{-4/3},\quad t\ge 1.
\end{equation}
For $n\le k$ and $\omega\in\Gamma_+$ we represent the jump of the resolvent   as
(cf. \cite[p.13]{EKT})
\[
\cR(\lambda+\I 0)- \cR(\lambda-\I 0)=\frac{(m+\lambda)|T(\theta)|^2}{-2\I\sin\theta}
[{\bf w}_k^+(\theta)\otimes{\bf w}_n^+(-\theta)+{\bf w}_k^-(\theta)\otimes{\bf w}_n^-(-\theta)].
\]
Inserting this into \eqref{sp-rep-sol} and integrating by parts, we get
\[
\left[ \E^{-\I t\D}P_c^+ \right]_{n,k}
=\left[{\mathcal P}^+(t)\right]_{n,k}+ \left[{\mathcal P}_-(t)\right]_{n,k},
\]
where
\begin{align}\nonumber
&\left[{\mathcal P}^{\pm}(t)\right]_{n,k}:=\frac{\I}{4\pi t}\int_{-\pi}^{\pi}\E^{-\I t g(\theta)}
\frac{d}{d\theta}\Big[\frac{(m+\lambda(\theta))|T(\theta)|^2}{\sin\theta}\E^{\pm\I\theta(k-n)}
{\bf h} _k^{\pm}(\theta)\otimes{\bf h}_n^{\pm}(-\theta)\Big]d\theta\\
\label{I-est}
&=\frac{1}{4\pi t}\int_{-\pi}^{\pi}\E^{-\I t (g(\theta)\mp\frac{k-n}t)}
\big(\mp(k-n)+\I\frac{d}{d\theta}\big)
\frac{(m+\lambda(\theta))|T(\theta)|^2}{\sin\theta}{\bf h}_k^{\pm}(\theta)\otimes{\bf h}_n^{\pm}(-\theta).
\end{align}
First, note that $T(\theta){\bf h}_p^{\pm}(\theta)\in\mathcal A$ if $q\in\ell^1_1$, and 
\begin{equation}\label{Thp}
\Vert T(\cdot){\bf h}_p^{\pm}(\cdot)\Vert_{\mathcal A}\le C, \quad\forall p\in \Z.
\end{equation}
Indeed, for $\pm p\ge 0$ it follows from \eqref{est10} and Theorem \ref{thm:scat}, and for $\pm p<0$ from the
scattering relation
\begin{equation}\label{scat-rel1} 
T(\theta){\bf h}_p^{\pm}(\theta)=R^{\mp}(\theta)h^{\mp}_p(\theta)\E^{\mp 2\I p\theta}+h_p^{\mp}(-\theta)
\end{equation}
(see \eqref{scat-rel}). Further,  representation \eqref{B} and the bounds \eqref{est3}--\eqref{estC} imply
\begin{equation} \label{alg-dif}
\frac{d}{d \theta} {\bf h}^\pm_p(\theta)
\in \mathcal A \quad \mbox{if} \quad q\in\ell^1_2.
\end{equation}
Therefore,
$\frac{d}{d\theta} W(\theta):=W'(\theta)\in \mathcal A$. Since
in the non-resonant case $W^{-1}(\theta)\in\mathcal A$, we also infer
\begin{equation}\label{imp12}
T'(\theta),\qquad  (R^\pm(\theta))'\in\mathcal A
\end{equation}
by Wiener's lemma. For the derivatives of ${\bf h}_p^\pm$ bounds of the type \eqref{est10} hold. Namely,
\begin{equation}\label{est11}
\Vert{\frac{d}{d \theta} {\bf h}^\pm_p}(\cdot)\Vert_{\mathcal A}
\le C ~~{\rm for}~~ \pm p\ge 0.
\end{equation}
Denote
\[
{\bf a}^\pm_p(j):=\sum_{s=j}^{\pm\infty}|a^\pm_{p,s}|,\qquad 
{\bf b}^\pm_p(j):=\sum_{s=j}^{\pm\infty}|b^\pm_{p,s}|.
\]
From \eqref{est3} it follows  that if $q\in \ell_2^1$,
then $a^\pm_{p,s},~b^\pm_{p,s}\in \ell_1^1(\Z_\pm)$ for any fixed $p$. Hence,
\begin{equation}\label{proper4}
{\bf a}^\pm_p(\cdot),~ {\bf b}^\pm_p(\cdot)\in\ell^1(\Z_\pm).
\end{equation}
Based on this observation we prove the following
\begin{lemma}\label{T-est}
Let $q\in \ell^1_2$ and  $W(0)W(\pi)\neq 0$. Then
\begin{equation}\label{Test1}
\Big\Vert\frac{T(\theta){\bf h} _p^\pm(\theta)}{\sin\theta}\Big\Vert_{\mathcal A}\le C (1 + |p|),\quad p\in \Z.
\end{equation}
\end{lemma}
\begin{proof}
Note that  $\displaystyle\frac{T(\theta)}{\sin\theta}=\displaystyle\frac{2\I}{(m+\lambda)W(\theta)}$ by \eqref{TR-def}. Hence, for $p\in\Z_\pm\cup\{0\}$ the
bound \eqref{Test1} follows from \eqref{est10} and Theorem~\ref{thm:scat}.
Consider the case $p\in \Z_\mp\cup\{0\}$. The scattering relations \eqref{scat-rel} imply
\begin{equation}\label{proper8}
T(\theta){\bf h} _p^\pm(\theta)
= (R^\mp(\theta)+1){\bf h}_p^\mp(\theta)\E^{\mp 2\I p\theta}-
({\bf h}_p^\mp(\theta)-{\bf h}_p^\mp(-\theta))\E^{\mp 2\I p\theta} + {\bf h}_p^\mp(-\theta)(1 -\E^{\mp 2\I p\theta}).
\end{equation}
Using \eqref{B}, \eqref{a-1}, \eqref{b-1}, and \eqref{ab2}, we obtain
\begin{align*}
&\frac{{\bf h}_{p,1}^\mp(\theta)-{\bf h}_{p,1}^\mp(-\theta)}{\sin\theta}=
\varkappa_{p}^{\mp}\sum\limits_{s=\mp 1}^{\mp\infty}
a^\mp_{p,s}\frac{\E^{\mp\I s\theta}-\E^{\pm\I s\theta}}{\sin\theta}\\
&=\mp 2\I\varkappa_{p}^{\mp}\sum\limits_{s=\mp 1}^{\mp \infty}a^\mp_{p,s}\times\left\{\begin{array}{ll}
(\E^{-\I (s-1)\theta}+...+\E^{-2\I \theta}+1+\E^{2\I \theta}+...+\E^{\I (s-1)\theta})~~{\rm for}~{\rm odd}~s\\
(\E^{-\I (s-1)\theta}+...+\E^{-\I \theta}+\E^{\I \theta}+...+\E^{\I (s-1)\theta})~~{\rm for}~{\rm even}~s
\end{array}\right.\\
&=\mp 2\I\varkappa_{p}^{\mp}\sum\limits_{j=-\infty}^{\infty}\Big(a^\mp_{p,\mp |j|\mp 1}+a^\mp_{p,\mp |j|\mp 3}
+a^\mp_{p,\mp |j|\mp 5}+....\Big)\E^{\I j\theta},
\end{align*}
where $\varkappa_{p}^{+}=(A_p^{+})^{-1}_{11}=1$, $\varkappa_{p}^{-}=(A_p^{-})^{-1}_{11}=1/(1-q_p)$.
Similarly,
\begin{align*}
&\frac{{\bf h}_{p,2}^\mp(\theta)-{\bf h}_{p,2}^\mp(-\theta)}{\sin\theta}=\frac{\varkappa_{p}^{\mp}}{m+\lambda}\Big[\mp 2\I\pm2\I b^{\mp}_{p,\pm 1}+
\sum\limits_{s=0}^{\mp\infty}b^\mp_{p,s}\frac{\E^{\mp\I s\theta}-\E^{\pm\I s\theta}}{\sin\theta}\Big]\\
&=\mp\frac{2\I\varkappa_{p}^{\mp}}{m+\lambda}\Big[1-b^{\mp}_{p,\pm 1}+\sum\limits_{j=-\infty}^{\infty}\Big(b^\mp_{p,\mp |j|\mp 1}+b^\mp_{p,\mp |j|\mp 3}
+b^\mp_{p,\mp |j|\mp 5}+....\Big)\E^{\I j\theta}\Big]
\end{align*}
Property \eqref{proper4} then  implies
\begin{equation}\label{proper5}
\Big\Vert\frac{{\bf h}_p^\mp(\theta)-{\bf h}_p^\mp(-\theta)}{\sin\theta}\Big\Vert_{\mathcal A}
\le C ,\quad p\in\Z_\mp.
\end{equation}
In particular,
\[
\frac{u_0^\mp(\theta)\!-u_0^\mp(-\theta)}{\sin\theta}=\frac{{\bf h}_{0,1}^\mp(\theta)\!-{\bf h}_{0,1}^\mp(-\theta)}{\sin\theta},\qquad
\frac{v_{1}^\mp(\theta)\!-v_{1}^\mp(-\theta)}{\sin\theta}
=\frac{\E^{\mp\I\theta}{\bf h}_{1,2}^\mp(\theta)\!-\E^{\pm\I\theta}{\bf h}_{1,2}^\mp(-\theta)}{\sin\theta}
\in \mathcal A,
\]
which implies that
\begin{equation}\label{proper9}
\frac{R^\mp(\theta)+1}{\sin\theta}=\frac{1}{W(\theta)}\frac{W(\theta)+W^{+}(\theta)}{\sin\theta }
\in\mathcal A.
\end{equation}
Finally,
\begin{equation}\label{proper10}
\Big\Vert\frac{1 -\E^{\pm 2\I p\theta}}{\sin\theta}\Big\Vert_{\mathcal A}\leq  2|p|.
\end{equation}
Substituting \eqref{proper5}, \eqref{proper9}, and \eqref{proper10} into \eqref{proper8} we get \eqref{Test1}.
\end{proof}
Now we return to the representation \eqref{I-est}.
Let $|k|\le |n|$. In this case $|k-n|\le 2|n|$, and
applying \eqref{Thp} and \eqref{Test1} to the factors $T(-\theta){\bf h}_n^{\pm}(-\theta)$ and
$T(\theta){\bf h}_k^\pm(\theta)/\sin\theta$, respectively, we obtain
\begin{equation}\label{proper11}
\Big\Vert(k-n)\frac{|T(\theta)|^2{\bf h}_k^\pm(\theta)\otimes{\bf h}_n^\pm(-\theta)}
{\sin\theta}\Big\Vert_{\mathcal A}\le C(1+|n|)(1+|k|).
\end{equation}
The case $|n|\le |k|$ is handled similarly.
Furthermore, applying \eqref{Test1}
to both $T(-\theta){\bf h} _n^\pm(-\theta)/\sin\theta$ 
and $T(\theta){\bf h} _k^\pm(\theta)/\sin\theta$ we obtain
\begin{equation}\label{proper12}
\Big\Vert\frac{|T(\theta)|^2{\bf h}_k^\pm(\theta)\otimes{\bf h}_n^\pm(-\theta)}
{\sin^2\theta}\Big\Vert_{\mathcal A}\le C (1+|n|)(1+|k|).
\end{equation}
To complete the proof  we need one more property.
\begin{lemma}\label{lem:proizv} Let $q\in\ell_2^1$ and $W(0)W(\pi)\neq 0$. Then 
\begin{equation}\label{proper13}
\Big\Vert \frac{d}{d\theta}(T(\theta){\bf h}_p^\pm(\theta))\Big\Vert_{\mathcal A}\leq C(1+|p|),\quad p\in\Z.
\end{equation}
\end{lemma}
\begin{proof} Since $T^\prime(\theta)$ 
are elements of $\mathcal A$ for $q\in\ell_2^1$ by \eqref{imp12},
then for $p\in \mathbb Z_\pm\cup\{0\}$ the statement of the Lemma is evident in view of \eqref{est11}.
To get it for $p\in \mathbb Z_\mp$ we use \eqref{imp12}, \eqref{est11}, and the formula
\[
\frac{d}{d\theta}(T(\theta){\bf h}_p^\pm(\theta))=
\frac{d}{d\theta}\left(R^\mp(\theta){\bf h}_p^\mp(\theta)\right)\,\E^{\mp 2\I p\theta} 
\mp 2\I p\, \E^{\pm 2\I p\theta}R^\mp(\theta){\bf h}_p^\mp(\theta) +\frac{d}{d\theta}{\bf h}_p^\mp(-\theta).
\]
\end{proof}
Now \eqref{Test1}, \eqref{proper11}, \eqref{proper12} and \eqref{proper13} imply
\begin{equation}\label{proper15}
\Vert\big(\mp(k-n)+\I\frac{d}{d\theta}\big)
\frac{|T(\theta)|^2}{\sin\theta}{\bf h} _k^{\pm}(\theta)\otimes{\bf h}_n^{\pm}(-\theta)\Vert_{\mathcal A}
\le  C (1+|n|)(1+|k|).
\end{equation}
Finally, we split the domain of integration in \eqref{I-est} into regions where
either the second or third derivative of the phase is nonzero. Then 
Lemma~\ref{lem:vC} together with \eqref{proper15} imply \eqref{HP-n1}.
\begin {theorem}\label{2-newn}
Let $q\in\ell^1_2$. Then in the non-resonant case the asymptotics  \eqref{as2-new} hold.
\end{theorem}
We consider the case $n\le k$ and obtain the asymptotics of type \eqref{as2-new} for ${\mathcal P}^+(t)$ defined in \eqref{I-est}. Namely, we should prove that
\begin{equation}\label{cP-est}
\Vert {\mathcal P}^+(t)\Vert_{\bl^2_\sigma\to \bl^2_{-\sigma}}\le C(t^{-3/2}),\quad t\to\infty,\quad\sigma>3/2.
\end{equation}
As in the proof of Theorem~\ref{end1} (ii) we consider the integrals over 
${\bf J}_{\pm}$ and over $\bf J$ separately. Namely,
taking into account  the scattering relation \eqref{scat-rel1}, we split  ${\mathcal P}^+(t)$ according to
\begin{equation}\label{MMMM}
{\mathcal P}^+(t)={\mathcal M}(t)+\sum\limits_{\pm}
\Big[{\mathcal M}_1^{\pm}(t)+{\mathcal M}_2^{\pm}(t)+{\mathcal M}_3^{\pm}(t)+{\mathcal M}_4^{\pm}(t)+{\mathcal M}_5^{\pm}(t)\Big],
\end{equation} 
where
\[
[{\mathcal M}(t)]_{n,k}=\frac{1}{4\pi t}\int_{\bf J}\E^{-\I t \Phi_v(\theta)}
\big(n-k+\I\frac{d}{d\theta}\big)
\frac{(m+\lambda)|T(\theta)|^2}{\sin\theta}{\bf h} _k^+(\theta)\otimes{\bf h}_n^+(-\theta),
\]
and
\begin{equation}\label{Mj}
[{\mathcal M}_j^\pm(t)]_{n,k}= \frac{1}{4\pi t}\int_{{\bf J}_{\pm}}
\E^{-\I t\Phi_{v_j}(\theta)} Z_{n,k}^j(\theta) d\theta,\quad j=1,...,5.
\end{equation}
Here we set  $\Phi_{v_j}(\theta)=g(\theta)-v_j\theta$ with
\[
v_1=v=\frac{k-n}t, \quad v_2=v_3=\frac{k+n}t,  \quad v_4=-\frac{k+n}t, \quad  v_5=\frac{n-k}t, 
\] 
and
\[
Z_{n,k}^1(\theta)=\left\{\begin{array}{ll}
\big(n-k+\I\frac{d}{d\theta}\big)
\frac{(m+\lambda)|T(\theta)|^2}{\sin\theta}{\bf h} _k^+(\theta)\otimes{\bf h}_n^+(-\theta),~~~&0\leq n\le k,\\\\
\big(n-k+\I\frac{d}{d\theta}\big)
\frac{(m+\lambda)T(\theta)}{\sin\theta}{\bf h} _k^+(\theta)\otimes{\bf h}_n^-(\theta),&n< 0\le k,\\\\
\big(n-k+\I\frac{d}{d\theta}\big)\frac{m+\lambda}{\sin\theta}
{\bf h} _k^-(-\theta)\otimes{\bf h}_n^-(\theta),&n\leq k< 0,
\end{array}\right.
\]
\smallskip
\[
Z_{n,k}^2(\theta)= \left\{\begin{array}{ll} 0,\quad 0\leq n\le k,\\\\
\!\!\big(\!-n-k+\I\frac{d}{d\theta}\big)
\frac{(m+\lambda)T(\theta)}{\sin\theta}R^{-}(-\theta){\bf h} _k^+(\theta)\otimes{\bf h}_n^-(-\theta),\quad n< 0\le k,~~|n|>k,\\\\
\!\!\big(\!-n-k+\I\frac{d}{d\theta}\big)\frac{(m+\lambda)}{\sin\theta}
R^{-}(-\theta){\bf h} _k^-(-\theta)\otimes{\bf h}_n^-(-\theta),\quad n\leq k< 0,
\end{array}\right.
\]
\smallskip
\[
Z_{n,k}^3(\theta)= \left\{\begin{array}{ll} 0,\quad 0\leq n\le k\quad {\rm and}\quad  n\le k< 0,\\\\
\big(-n-k+\I\frac{d}{d\theta}\big)
\frac{(m+\lambda)T(\theta)}{\sin\theta}R^{-}(-\theta){\bf h} _k^+(\theta)\otimes{\bf h}_n^-(-\theta),~~ n< 0\le k,~~|n|\le k,
\end{array}\right.
\]
\smallskip
\[
Z_{n,k}^4(\theta)= \left\{\begin{array}{ll} 0,\quad 0\leq n\le k \quad {\rm and}\quad n< 0\le k,\\\\
\big(k+n+\I\frac{d}{d\theta}\big)\frac{(m+\lambda)R^{-}(\theta)}{\sin\theta}
{\bf h} _k^-(\theta)\otimes{\bf h}_n^-(\theta),\quad n\leq k\leq 0,
\end{array}\right.
\]
\smallskip
\[
Z_{n,k}^5(\theta)= \left\{\begin{array}{ll} 0,\quad 0\leq n\le k \quad {\rm and} \quad n< 0\le k,\\\\
\big(k-n+\I\frac{d}{d\theta}\big)\frac{(m+\lambda)|R^{-}(\theta)|^2}{\sin\theta}
{\bf h} _k^-(\theta)\otimes{\bf h}_n^-(-\theta),\quad n\leq k< 0.
\end{array}\right.
\]
\smallskip\\
Note, that the sign of  each $v_j=v_j(n,k)$ in the representation  \eqref{Mj} for $[{\mathcal M}^{\pm}_j(t)]_{n,k}$ does not depend on $n,k$. Namely,
for $t>0$ one has
\[
\left\{\begin{array}{ll} 
v_j\ge 0 \quad {\rm for}\quad j=1,3,4,\\
v_j\le 0 \quad{\rm for}\quad j=2,5.
\end{array}\right.
\]
Lemma~\ref{lem:vC} with $s=2$ and \eqref{proper15} imply
\[
|[{\mathcal M}(t)]_{n,k}|\le Ct^{-3/2}(1+|n|)(1+|k|),\quad n,k\in\Z, \quad t\ge 1.
\]
Hence,  the asymptotics  of type \eqref{cP-est} for ${\mathcal M}(t)$ follow. 
Further, Proposition \ref{Jost-sol} (i)  implies
\[
|\frac{d^p}{d\theta^p}T(\theta))|,~|\frac{d^p}{d\theta^p}R^\pm(\theta)|\le C,\quad 0\le p\le 2,\quad \theta\in {\bf J}_{\pm}.
\]
Respectively, 
\begin{equation}\label{diff-Z}
|Z_{n,k}^j(\theta)|+|\frac{d}{d\theta}Z_{n,k}^j(\theta)|\le C(1+\max\{|n|,|k|\}),
\quad n,k\in\Z,\quad \theta\in {\bf J}_{\pm},\quad j=1,\dots 4.
\end{equation}
The operators  ${\mathcal M}^{\pm}_j(t)$ with  $j=1,3,4$ and the operators ${\mathcal M}^{\mp}_j(t)$ with  $j=2,5$
are estimated  in the same way  
as the operators ${\mathcal K}^{\pm}(t)$ in the proof of Theorem \ref{end1}.
Namely, applying integration by parts, we obtain
\[
|[{\mathcal M}^-_j(t)]_{n,k}|\le C t^{-2}(1+|n|)(1+|k|),\quad n,k\in\Z,\quad j=1,3,4, \quad t\ge 1.
\]
\[
|[{\mathcal M}^+_j(t)]_{n,k}|\le C t^{-2}(1+|n|)(1+|k|),\quad n,k\in\Z,\quad j=2,5, \quad t\ge 1.
\]
Hence,   asymptotics  of type \eqref{cP-est} for ${\mathcal M}^-_j(t)$ with  $j=1,3,4$ and for ${\mathcal M}^+_j(t)$ with  $j=2,5$ follow. 
Further,   applying  \cite [Lemma 6.3] {KT1} with  $p=1$, we obtain  the asymptotics  for ${\mathcal M}^+_j(t)$ with  $j=1,3,4$.  The asymptotics for
${\mathcal M}^-_j(t)$ with  $j=2,5$ follow by the same   lemma with  ${\bf J}_+$ replaced by ${\bf J}_-$.
\end{proof}
\appendix
\section{The calculation of $\tilde T(0)$}
\label{App}
Representation \eqref{B} implies
\begin{align*}
(m+\lambda)W(\theta)&=\tilde u^+_{n}(z)\tilde w^-_{n+1}(z)-\tilde u^-_{n}(z)\tilde w^+_{n+1}(z)
=\frac1z\big[1+\sum\limits_{k=0}^{\infty}a^{+}_{n,k}z^k\big]
\big[z-1+\sum\limits_{k=0}^{\infty}b^{-}_{n+1,-k}z^k\big]\\
&-\frac{z}{(1\!-\!q_n)(1\!-\!q_{n+1})}\big[1+\sum\limits_{k=0}^{\infty}a^{-}_{n,-k}z^k\big]
\big[\frac 1z-1+\sum\limits_{k=-1}^{\infty}b^{+}_{n+1,k}z^k\big]\\
&=\frac{A_{-1}}z+A_{0}+A_{1}z+...,\quad z\to 0,
\end{align*}
where
\[
A_{-1}=(1+a^{+}_{n,0})(1+b^-_{n+1,0})=(1+a^{+}_{0,0})(b^-_{1,0}-1)
\]
does not depend on $n$.
Assume that $b^-_{1,0}=1$. Then \eqref{ab2} implies that
$b_{0,0}^--\tilde q_{0}(b_{0,0}^--1)=1$. Hence,  $b_{0,0}^-=1$.
Repeating this, we obtain that $b^-_{n,0} =1$ for all $n\le 1$, which  contradicts
\eqref{est3}--\eqref{estC}.
Similarly, if $a^{+}_{0,0}=-1$ then 
$a^{+}_{n,0}=-1$ for all $n\ge 0$ by  \eqref{a-1}, which again contradicts \eqref{est3}--\eqref{estC}.
Therefore,  $A_{-1}\not =0$.
Moreover,
\[
\tilde T(z)=\frac{2\I\sin\theta}{(m+\lambda)W(\theta)}\sim \frac{1-z^2}{A_{-1}+A_{0}z+A_{1}z^2+...},\quad z\to 0.
\]
Hence,
\[
\tilde T(0)=\frac 1{A_{-1}}<\infty, \quad \tilde T'(0)=\frac{-A_{0}}{A_{-1}^2}<\infty.
\]


\end{document}